\documentclass[a4paper,12pt]{amsart}
\usepackage{amsmath, amsthm, amssymb, amsopn, amscd, mathtools, array, enumerate, amsfonts, graphics, hyperref, wasysym, chngcntr, linguex, graphicx}

\newtheorem{theorem}{Theorem}[section]
\newtheorem{corollary}[theorem]{Corollary}
\newtheorem{lemma}[theorem]{Lemma}

\newtheorem{lemma and definition}[theorem]{Lemma and Definition}
\newtheorem{proposition}[theorem]{Proposition}
\newtheorem{definition}[theorem]{Definition}
\newtheorem{exam}[theorem]{Example}

\newtheorem{remark}[theorem]{Remark}

\newtheorem{discussion}[theorem]{Discussion}

\newtheorem{the construction}[theorem]{THE CONSTRUCTION}
\newtheorem{pd}[theorem]{Proposition-Definition}

\makeatletter
\newcommand*\rel@kern[1]{\kern#1\dimexpr\macc@kerna}
\newcommand*\widebar[1]{%
	\begingroup
	\def\mathaccent##1##2{%
		\rel@kern{0.8}%
		\overline{\rel@kern{-0.8}\macc@nucleus\rel@kern{0.2}}%
		\rel@kern{-0.2}%
	}%
	\macc@depth\@ne
	\let\math@bgroup\@empty \let\math@egroup\macc@set@skewchar
	\mathsurround\z@ \frozen@everymath{\mathgroup\macc@group\relax}%
	\macc@set@skewchar\relax
	\let\mathaccentV\macc@nested@a
	\macc@nested@a\relax111{#1}%
	\endgroup
}
\makeatother

\newcommand{\field}[1]{\mathbb{#1}}

\newcommand{\N }{\field{N}}

\DeclareMathOperator{\im}{Im}
\DeclareMathOperator{\Ie}{I}

\DeclareMathOperator{\supp}{Supp}
\DeclareMathOperator{\crk}{crk}

\DeclareMathOperator{\coker}{Coker}
\DeclareMathOperator{\kerr}{Ker}
\DeclareMathOperator{\ass}{Ass}
\DeclareMathOperator{\ann}{Ann}
\DeclareMathOperator{\homm}{Hom}
\DeclareMathOperator{\spec}{Spec}

\title{A Generalization of the Katzman-Zhang Algorithm}
\author{Mehmet Yesil}
\address{Department of Pure Mathematics, The University of Sheffield, Hicks Building, Sheffield S3 7RH, UK}
\address{Department of Mathematics, Batman University, Batman, Turkey}
\email{mehmet-yesil@outlook.com}

\begin{document}
\maketitle

\begin{abstract}
	In this paper, we study the notion of special ideals. We generalize the results on those as well as the algorithm obtained for finite dimensional power series rings by Mordechai Katzman and Wenliang Zhang to finite dimensional polynomial rings. 
\end{abstract}

\section{Introduction}
Throughout this paper our general assumption on $R$ is to be a commutative Noetherian regular ring of prime characteristic $p$. Let $e$ be a positive integer. Let $f: R\rightarrow R$ be the Frobenius homomorphism defined by $f(r)=r^{p}$ for all $r \in R$, whose $e$-th iteration is denoted by $f^{e}$. Let $M$ be an $R$-module. $F_{*}^{e}M=\{F_{*}^{e}m \mid m\in M \}$ denotes the Abelian group $M$ with the induced $R$-module structure via the $e$-th iterated Frobenius and it is given by
\begin{center}
	$rF_*^{e}m=F_*^{e}r^{p^e}m$ for all $m \in M$ and $r\in R$
\end{center}
In particular, $F_{*}^{e}R$ is the Abelian group $R$ with the induced $R$-module structure
\begin{center}
	$rF_*^{e}s=F_*^{e}r^{p^e}s $ for all $r,s\in R$.
\end{center}

An $e$-th Frobenius map on $M$ is an $R$-linear map $\phi:M \rightarrow F_{*}^{e}M$, equivalently an additive map $\phi:M\rightarrow M$ such that $\phi(rm)=r^{p^{e}}\phi(m)$ for all $r\in R$ and $m\in M$. Let $R[X;f^{e}]$ be the skew-polynomial ring whose multiplication is subject to the rule $Xr=f^{e}(r)X=r^{p^{e}}X$ for all $r\in R$. Notice that defining an $e$-th Frobenius map on $M$ is equivalent to endowing $M$ with a left $R[X;f^{e}]$-module structure extending the rule $Xm=\phi(m)$ for all $m\in M$.

This paper studies the notion of special ideals. It was introduced by R. Y. Sharp in \cite{RY}. For a left $R[\phi;f^{e}]$-module $M$, when $\phi$ is injective on $M$, he defines an ideal of $R$ to be $M$-special $R$-ideal if it is the annihilator of some $R[\phi;f^{e}]$-submodule of $M$ (cf. \cite[Section 1]{RY}). Later on, it was generalized by M. Katzman and used to study Frobenius maps on injective hulls in \cite{K1} and \cite{K5}. For a left $R[\phi;f^{e}]$-module $M$, Katzman defines an ideal of $R$ to be $M$-special if it is the annihilator of some $R[\phi;f^{e}]$-submodule of $M$ (cf. \cite[Section 6]{K1}). A special case of special ideals is when $R$ is local, $M$ is Artinian, and $\phi$ is injective. In this case, Sharp showed that the set of $M$-special ideals is a finite set of radicals, consisting of all intersections of the finitely many primes in it (\cite[Corollary 3.11]{RY}). It was also proved by F. Enescu and M. Hochster independently (\cite[Section 3]{EH}). When $R$ is complete local regular and $M$ is Artinian, the notion of special ideals becomes an important device to study Frobenius maps on injective hulls. In particular, since top local cohomology module of $R$ is isomorphic to the injective hull of the residue field of $R$, it provides an important insight to top local cohomology modules.

In the case that $R$ is a finite dimensional formal power series ring over a field of prime characteristic $p$, in \cite{K3}, M. Katzman and W. Zhang focus on the $M$-special ideals when $M$ is Artinian. In this case, they define the special ideals depending on the $R[\phi;f]$-module structures on $E^{\alpha}$, where $E$ is the injective hull of the residue field of $R$ and $\alpha$ is a positive integer. They define an ideal of $R$ to be $\phi$-special if it is the annihilator of an $R[\phi;f]$-submodule of $E^{\alpha}$, where $\phi=U^{t}T$ with $T$ is the natural Frobenius on $E^{\alpha}$ and $U$ is an $\alpha\times\alpha$ matrix with entries in $R$ (see Section \ref{section:Katzman-Zhang Algorithm}). Furthermore, they use Katzman's $\Delta^{e}$ and $\Psi^{e}$ functors, which are extensions of Matlis duality keeping track of Frobenius maps, to define $\phi$-special ideals equivalently to be the annihilators of $R^{\alpha}/W$ for some submodule $W$ satisfying $UW\subseteq W^{[p]}$, where $W^{[p]}$ is the submodule generated by $\{ w^{[p]}=(w_{1}^{p},\dots,w_{\alpha}^{p})^{t}\mid w=(w_{1},\dots,w_{\alpha})^{t}\in W\}$ (see Proposition \ref{art1}). Katzman and Zhang show that there are only finitely many $\phi$-special ideals $P$ of $R$ with the property that $P$ is the annihilator of an $R[\phi;f^{e}]$-submodule $M$ of $E^{\alpha}$ such that the restriction of $\phi^{e}$ to $M$ is not zero for all $e$, and introduce an algorithm for finding special prime ideals with this property in \cite{K3}. They first present the case $\alpha=1$, which was considered by M. Katzman and K. Schwede in \cite{K2} with a geometric language. Then they extend this to the case $\alpha>1$. 

In this paper, we adapt the equivalent definition of $\phi$-special ideals above to the polynomial rings, and for an $\alpha\times\alpha$ matrix $U$ we define $U$-special ideals to be the annihilators of $R^{\alpha}/W$ for some submodule $W$ of $R^{\alpha}$ satisfying $UW\subseteq W^{[p]}$. We generalize the results in \cite{K3} to the case that $R$ is a finite dimensional polynomial ring over a field of prime characteristic $p$, and show that there are only finitely many $U$-special ideals with some non degeneracy conditions (see Theorem \ref{last4}). We also present an algorithm for finding $U$-special prime ideals of polynomial rings. Furthermore, we consider the notion of $F$-finite $F$-modules, which is a prime characteristic extension of local cohomology modules introduced by G. Lyubeznik in \cite{L1}, and we show that our new algorithm gives a method for finding the prime ideals of $R$ such that $\crk(H_{IR_{P}}^{i}(R_{P})) \neq 0$ (see Definition \ref{crk} and Theorem \ref{m1}).

\section{Preliminaries}

In this section, we collect some notations and necessary background for this paper.

\subsection{The Frobenius Functor}
Let $M$ be an $R$-module. The Frobenius functor $F_{R}$ from the category of $R$-modules to itself is defined by $F_{R}(M):=F_{*}R\otimes_{R}M$ where $F_{R}(M)$ acquires its $R$-module structure via the identification of $F_{*}R$ with $R$. The resulting $R$-module structure on $F_{R}(M)$ satisfies
\[
s(F_{*}r\otimes m)=F_{*}sr\otimes m \text{ and } F_{*}s^{p}r\otimes m=F_{*}r\otimes sm
\] 
for all $r,s \in R$ and $m \in M$. The $e$-th iteration of $F_{R}$ is denoted by $F_{R}^{e}$, and it is clearly given by $F_{R}^{e}(M)=F_{*}^{e}R\otimes_{R}M$.

Regularity of $R$ implies that the Frobenius functor is exact.

\subsection{Lyubeznik's $F$-modules}
An $R$-module $\mathcal{M}$ is called to be an $F$-module if it is equipped with an $R$-module isomorphism $\theta : \mathcal{M} \rightarrow F_{R}(\mathcal{M})$ which we call the structure isomorphism of $\mathcal{M}$.

An $F$-module homomorphism is an $R$-module homomorphism $\phi: \mathcal{M} \rightarrow \mathcal{M'}$ such that the following diagram commutes
\[
\begin{CD}
\mathcal{M} @>\phi>> \mathcal{M'}\\
@V\theta VV @VV\theta' V\\ 
F_{R}(\mathcal{M}) @>>F_{R}(\phi)> F_{R}(\mathcal{M'})
\end{CD}
\]
where $\theta$ and $\theta'$ are the structure isomorphisms of $\mathcal{M}$ and $\mathcal{M'}$, respectively.

A generating morphism of an $F$-module $\mathcal{M}$ is an $R$-module homomorphism $\beta : M \rightarrow F_{R}(M)$, where $M$ is an $R$-module, such that $\mathcal{M}$ is the limit of the inductive system in top row of commutative diagram
\[
\begin{CD} 
M @>\beta >> F_{R}(M) @>F_{R}(\beta)>> F_{R}^{2}(M) @>F_{R}^{2}(\beta)>> \cdots\\
@V\beta VV @V F_{R}(\beta)VV @V F_{R}^{2}(\beta)VV\\
F_{R}(M) @>>F_{R}(\beta)> F_{R}^{2}(M) @>>F_{R}^{2}(\beta)> F_{R}^{3}(M) @>>F_{R}^{3}(\beta)> \cdots
\end{CD}
\]
and the structure isomorphism of $\mathcal{M}$ is induced by the vertical arrows in this diagram.

An $F$-module $\mathcal{M}$ is called $F$-finite if it has a generating morphism $\beta : M \rightarrow F_{R}(M)$ with $M$ a finitely generated $R$-module. In addition, if $\beta$ is injective, $M$ is called a root of $\mathcal{M}$ and $\beta$ is called a root morphism of $\mathcal{M}$.

\begin{exam} \label{t3}
Any $R$-module isomorphism $\phi:R \rightarrow F_{R}(R)$ makes $R$ into an $F$-module. In particular, the canonical isomorphism 
\[
\phi:R \rightarrow F_{*}R \otimes_{R}R=F_{R}(R) \text{ defined by } r\mapsto F_{*}r\otimes 1.
\]
In particular, $R$ is $F$-finite $F$-module. Therefore, by \cite[Proposition 2.10]{L1}, local cohomology modules $H_{I}^{i}(R)$ with support on an ideal $I\subseteq R$ are $F$-finite $F$-modules. Furthermore, by \cite[Proposition 2.3]{L1}, we have
\[
H_{I}^{i}(R)=\varinjlim(M\xrightarrow{\beta}F_{R}(M) \xrightarrow{F_{R}(\beta)} F_{R}^{2}(M) \xrightarrow{F_{R}^{2}(\beta)}\cdots)
\]
where $\beta : M \rightarrow F_{R}(M)$ is a root morphism.
\end{exam}

\subsection{$I_{e}(-)$ Operation and $\star$-closure}

In this subsection, we will give definitions of $\Ie_{e}(-)$ operation and $\star$-closure, and some properties of them. To do this we need the property that $F_{*}^{e}R$ are intersection flat $R$-modules for all positive integer $e$.

\begin{definition}
	An $R$-module $M$ is intersection flat\index{intersection flat modules} if it is flat and for all sets of $R$-submodules $\{N_{\lambda}\}_{\lambda \in \Lambda}$ of a finitely generated $R$-module $N$,
	$$M \otimes _{R} \bigcap _{\lambda \in \Lambda} N_{\lambda}= \bigcap _{\lambda \in \Lambda} (M \otimes _{R} N_{\lambda})$$
\end{definition}

\begin{definition}
	Let $I$ be an ideal of $R$. Ideal generated by the set $\{r^{p^{e}} \mid r\in I \}$ is called the Frobenius power of $I$ and denoted by $I^{[p^{e}]}$. Consequently, if $I=\langle r_{1}, \dots ,r_{n}\rangle$, then $I^{[p^{e}]}=\langle r_{1}^{p^{e}}, \dots ,r_{n}^{p^{e}}\rangle$.
\end{definition}

\begin{remark}
	Since intersection flat $R$-modules include $R$ and closed under arbitrary direct sum, free $R$-modules are intersection flat. For instance, when $R=\Bbbk[x_{1},\dots,x_{n}]$ over a field $\Bbbk$ of prime characteristic $p$, $F_{*}^{e}R$ are free, and so intersection flat. In addition, by \cite[Proposition 5.3]{K1}, when $R=\Bbbk[\![x_{1},\dots,x_{n}]\!]$ over a field $\Bbbk$ of prime characteristic $p$,  $F_{*}^{e}R$ are intersection flat. Because of regularity, these rings have the property that for any collection of ideals $\{A_{\lambda}\}_{\lambda \in \Lambda}$ of $R$,
	\[
	(\cap_{\lambda \in \Lambda}A_{\lambda})^{[p^{e}]}\cong F_{R}^{e}(\cap_{\lambda \in \Lambda}A_{\lambda})\cong \cap_{\lambda \in \Lambda}F_{R}^{e}(A_{\lambda}) \cong \cap_{\lambda \in \Lambda}A_{\lambda}^{[p^{e}]},
	\]
	and this is enough to define the minimal ideal $J \subseteq R$ with the property $A \subseteq J^{[p^{e}]}$.
\end{remark}

Henceforth $R$ will denote a ring with the property that $F_{*}^{e}R$ are intersection flat for all positive integer $e$.

\begin{pd} \cite[Section 5]{K1} Let $e$ be a positive integer.
	\begin{enumerate}
		\item For an ideal $A \subseteq R$ there exists a minimal ideal $J \subseteq R$ with the property $A \subseteq J^{[p^{e}]}$. We denote this minimal ideal by $\Ie_{e}(A)$\index{$\Ie_{e}(-)$ operation}.
		\item Let $u \in R$ be a non zero element and $A \subseteq R$ an ideal. The set of all ideals $B \subseteq R$ which contain $A$ and satisfy $uB \subseteq B^{[p^{e}]}$ has a unique minimal element. We call this ideal the star closure of $A$ with respect to $u$ and denote it by $A^{\star^{e}u}$\index{$\star$-closure}.
	\end{enumerate}
\end{pd}

\begin{definition}
	Given any matrix (or vector) $V$ with entries in $R$, we define $V^{[p^{e}]}$ to be the matrix obtained from $V$ by raising its entries to the $p^{e}$-th power. Given any submodule $K \subseteq R^{\alpha}$, we define $K^{[p^{e}]}$ to be the $R$-submodule of $R^{\alpha}$ generated by $\{ v^{p^{e}} \mid v \in K \}$.
\end{definition}

The Proposition-Definition below extends $\Ie_{e}(-)$-operation and $\star$-closure defined on ideals to submodules of free $R$-modules.

\begin{pd} Let $e$ be a positive integer.
	\begin{enumerate}
		\item Given a submodule $K \subseteq R^{\alpha}$ there exists a minimal submodule $L \subseteq R^{\alpha}$ for which $K \subseteq L^{[p^{e}]}$. We denote this minimal submodule $\Ie_{e}(K)$.
		\item Let $U$ be an $\alpha \times \alpha$ matrix with entries in $R$ and $V \subseteq R^{\alpha}$. The set of all submodules $K \subseteq R^{\alpha}$ which contain $V$ and satisfy $UK \subseteq K^{[p^{e}]}$ has a unique minimal element. We call this submodule the star closure of $V$ with respect to $U$ and denote it $V^{\star^{e}U}$.
	\end{enumerate}
\end{pd}

\begin{proof} For the proof of (\textit{1}) we refer to \cite[Section 3]{K3}. For the proof of (\textit{2}) we shall construct a similar method to that in \cite[Section 3]{K3}. Let $V_{0}=V$ and $V_{i+1}=I_{e}(UV_{i})+V_{i}$. Then $\{V_{i}\}_{i\geq 0}$ is an ascending chain and it stabilizes, since $R$ is Noetherian, i.e. $V_{j}=V_{j+k}$ fo all $k>0$ for some $j\geq 0$. Therefore, $V_{j}=I_{e}(UV_{j})+V_{j}$ implies $I_{e}(UV_{j}) \subseteq V_{j}$, and so $UV_{j} \subseteq V_{j}^{[p^{e}]}$. We show the minimality of $V_{j}$ by induction on $i$. Let $Z$ be any submodule of $R^{\alpha}$ containing $V$ with the property that $UZ \subseteq Z^{[p^{e}]}$. Then we clearly have $V_{0}=V \subseteq Z$, and suppose that $V_{i} \subseteq Z$ for some $i$. Thus, $UV_{i} \subseteq UZ \subseteq Z^{[p^{e}]}$, which implies $I_{e}(UV_{i}) \subseteq Z$ and so $V_{i+1} \subseteq Z$. Hence, $V_{j} \subseteq Z$. 
\end{proof}

For the calculation of $\Ie_{e}(-)$ operation, if $R$ is a free $R^{p^{e}}$-module, we first fix a free basis $\mathcal{B}$ for $R$ as an $R^{p^{e}}$-module, then every element $v \in R^{\alpha}$ can be expressed uniquely in the form $v=\sum_{b\in\mathcal{B}}u_{b}^{[p^{e}]}b$ where $u_{b} \in R^{\alpha}$ for all $b\in \mathcal{B}$.

\begin{proposition}\cite[Proposition 3.4]{K3} Let $e>0$.
	\begin{enumerate}
		\item For any submodules $V_{1}, \dots ,V_{n}$ of $R^{\alpha}$, $\Ie_{e}(V_{1}+ \cdots +V_{n})=\Ie_{e}(V_{1})+ \cdots +\Ie_{e}(V_{n})$.
		\item Let $\mathcal{B}$ be a free basis for $R$ as $R^{p^{e}}$-module. Let $v \in R^{\alpha}$ and $v=\sum_{b\in\mathcal{B}}u_{b}^{[p^{e}]}b$ be the unique expression for $v$ where $u_{b} \in R^{\alpha}$ for all $b\in \mathcal{B}$. Then $\Ie_{e}(\langle v\rangle)$ is the submodule of $R^{\alpha}$ generated by $\{u_{b}\mid b\in \mathcal{B}\}$. 
	\end{enumerate}
	
\end{proposition}

The behaviour of the $\Ie_{e}(-)$ operation under localization is very crucial for our results. The following lemma shows that it commutes with localization.

\begin{lemma}\cite[Lemma 2.5]{K4} \label{ie} Let $\mathcal{R}$ be a localization of $R$ or a completion at a prime ideal. For all $e\in\N$, and all submodules $K \subseteq R^{\alpha}$, $\Ie_{e}(K \otimes_{R} \mathcal{R})$ exists and equals to $\Ie_{e}(K) \otimes_{R} \mathcal{R}$.
\end{lemma}

\begin{lemma}\label{sp2}
	Let $U$ be a non-zero $\alpha \times \alpha$ matrix with entries in $R$ and $K \subseteq R^{\alpha}$ a submodule. For any prime ideal $P\subseteq R$,
	\[
	(\widehat{K_{P}})^{\star^{e}U}=\widehat{(K^{\star^{e}U})_{P}}.
	\]
\end{lemma}

\begin{proof}
	Define inductively $K_{0}=K$ and $K_{i+1}=I_{e}(UK_{i})+K_{i}$, and also $L_{0}=\widehat{K_{P}}$ and $L_{i+1}=I_{e}(UL_{i})+ L_{i}$ for all $i \geq 0$. Since $\Ie_{e}(-)$ operation commutes with localization and completion, an easy induction shows that $L_{i}=\widehat{(K_{i})_{P}}$, and the result follows.
\end{proof}

\section{The Katzman-Schwede Algorithm}
\label{section:Katzman-Schwede Algorithm}
The purpose of this section is to redefine the algorithm described in \cite{K2} with a more algebraic language and show that it commutes with localization. Let $R=\Bbbk[x_{1},\dots,x_{n}]$ be a polynomial ring over a field of characteristic $p$ and $e$ be a positive integer.

\begin{definition}
	For any $R$-linear map $\phi : F_{*}^{e}R \rightarrow R$, we say that an ideal $J \subseteq R$ is $\phi$-compatible\index{compatible ideal} if $\phi(F_{*}^{e}J) \subseteq J$.
\end{definition}

Given $\phi$ which is compatible with $J$ as above definition, there is always a commutative diagram
$$\begin{array}[c]{ccc}
F_{*}^{e}R&\stackrel{\phi}{\longrightarrow}&R\\
\downarrow\scriptstyle{}&&\downarrow\scriptstyle{}\\
F_{*}^{e}(R/J)&\stackrel{\phi'}{\longrightarrow}&R/J
\end{array}$$
where the vertical arrows are the canonical surjections.

\begin{lemma}\cite[Lemma 2.4]{K2}
	Assuming a commutative diagram as above, the $\phi$-compatible ideals containing $J$ are in the bijective correspondence with the $\phi'$-compatible ideals of $R/J$, where $\phi'$ is the induced map $F_{*}^{e}(R/J) \stackrel{\phi'}{\longrightarrow} R/J$ as in above diagram.
\end{lemma}

Next we will explain the $F_{*}^{e}R$-module structure of $\homm_{R}(F_{*}^{e}R,R)$, which is crucial for our computational techniques in this paper.

\begin{remark} \label{rfree}
	Let $\mathcal{C}$ be a base for $\Bbbk$ as a $\Bbbk^{p^{e}}$-vector space which includes the identity element of $\Bbbk$. It is well known that $F_{*}^{e}R$ is a free $R$-module with the basis set
	\[
	\mathcal{B}=\{ F_{*}^{e}\lambda x_{1}^{\alpha_{1}} \dots x_{n}^{\alpha_{n}} \mid 0 \leq \alpha_{1}, \dots,\alpha_{n} < p^{e}, \lambda \in \mathcal{C} \}.
	\]
\end{remark}

\begin{lemma} \cite[cf. Example 3.0.5]{BKS}\label{trace}
	Let $\pi_{e}:F_{*}^{e}R \rightarrow R$ be the projection map onto the free summand $RF_{*}^{e} x_{1}^{p^{e}-1} \dots x_{n}^{p^{e}-1}$. Then $\homm_{R}(F_{*}^{e}R,R)$ is generated by $\pi_{e}$ as an $F_{*}^{e}R$-module.
\end{lemma}

\begin{proof}
	For each basis element $F_{*}^{e} \lambda x_{1}^{\alpha_{1}} \dots x_{n}^{\alpha_{n}} \in \mathcal{B}$, the projection map onto the free summand $RF_{*}^{e} \lambda x_{1}^{\alpha_{1}} \dots x_{n}^{\alpha_{n}}$ is defined by the rule $$F_{*}^{e}z.\pi_{e}(-)=\pi_{e}(F_{*}^{e}z.-),$$ where $z=\lambda^{-1}x_{1}^{p^{e}-1-\alpha_{1}} \dots x_{n}^{p^{e}-1-\alpha_{n}}$. Since we can obtain all of the projections in this way, the map
	\[
	\Phi:F_{*}^{e}R \rightarrow \homm_{R}(F_{*}^{e}R,R) \text{ defined by }\Phi(F_{*}^{e}u)=\phi_{u},
	\]
	where $\phi_{u}:F_{*}^{e}R \rightarrow R$ is the $R$-linear map $\phi_{u}(-)=\pi_{e}(F_{*}^{e}u-)$, is surjective. On the other hand, if $\Phi(F_{*}^{e}u)=0$ for some $u\in R$, then we have
	\begin{center} $\phi_{u}(F_{*}^{e}r)=\pi_{e}(F_{*}^{e}ur)=F_{*}^{e}u.\pi_{e}(F_{*}^{e}r)=0$ for all $r\in R$.
	\end{center} This means that $F_{*}^{e}u$ must be zero, and so $\Phi$ is injective.  Hence, $\Phi$ is an $F_{*}^{e}R$ isomorphism. In other words, $\pi_{e}$ generates $\homm_{R}(F_{*}^{e}R,R)$ as an $F_{*}^{e}R$-module.
\end{proof}

\begin{definition}
	Let the notation and situation be as in Lemma \ref{trace}. We call the map $\pi_{e}$ the trace map on $F_{*}^{e}R$, or just the trace map when the content is clear. 
\end{definition}

Next lemma provides an important property of the trace map $\pi_{e}$ which gives the relation between elements of $\homm_{R}(F_{*}^{e}R,R)$ and $\Ie_{e}(-)$ operation (cf. \cite[Claim 6.2.2]{BKS}).

\begin{lemma} \label{fa1}
	Let $A$ and $B$ be ideals of $R$. Then $\pi_{e}(F_{*}^{e}A) \subseteq B$ if and only if $A \subseteq B^{[p^{e}]}$.
\end{lemma}

\begin{proof}
	$(\Rightarrow)$ Since $R$ Noetherian, $A$ is finitely generated, and since $\pi_{e}$ is $R$-linear we may assume that $A$ is a principal ideal, i.e. $A=aR$ for some $a \in R$. Now since $F_{*}^{e}R$ is a free $R$-module with basis $\mathcal{B}$ as in Remark \ref{rfree}, $F_{*}^{e}a=\sum_{i}r_{i}F_{*}^{e}g_{i}$ for some $r_{i} \in R$ and $F_{*}^{e}g_{i} \in \mathcal{B}$. On the other hand, by Lemma \ref{trace}, $\pi_{e}(F_{*}^{e}z_{i}a)=r_{i}$ for some $z_{i}\in R$. This implies that $\pi_{e}(F_{*}^{e}Ra)=\langle r_{i} \rangle$. Then by the assumption $\pi_{e}(F_{*}^{e}A)=\langle r_{i}\rangle \subseteq B$, and since $F_{*}^{e}a=F_{*}^{e}\sum_{i}r_{i}^{p^{e}}g_{i}$ we have $a=\sum_{i}r_{i}^{p^{e}}g_{i} \in B^{[p^{e}]}$. Hence, $A \subseteq B^{[p^{e}]}$.
	
	$(\Leftarrow)$ Assume first that $A \subseteq B^{[p^{e}]}$ which implies that $F_{*}^{e}A \subseteq F_{*}^{e}B^{[p^{e}]}$. Therefore,
	\[
	\pi_{e}(F_{*}^{e}A) \subseteq \pi_{e}(F_{*}^{e}B^{[p^{e}]})=\pi_{e}(BF_{*}^{e}R)=B\pi_{e}(F_{*}^{e}R)\subseteq B.
	\]
\end{proof}

\begin{corollary}\label{fac}
	Let $A$ be an ideal of $R$, and let $\phi \in \homm_{R}(F_{*}^{e}R,R)$ be such that $\phi(-)=\pi_{e}(F_{*}^{e}u-)$ for some $u\in R$. Then $\phi(F_{*}^{e}A)=\pi_{e}(F_{*}^{e}uA)=I_{e}(uA)$ and $\star$-closure of $A$ gives the smallest $\phi$-compatible ideal containing $A$.	
\end{corollary}

\begin{proof}
	Since $uA \subseteq I_{e}(uA)^{[p^{e}]}$, the first claim follows from Lemma \ref{fa1}. The second claim follow from the fact that
	\begin{align*}
		A \text{ is } \phi-\text{compatible} &\Leftrightarrow\phi(F_{*}^{e}A)=\pi_{e}(F_{*}^{e}uA)=I_{e}(uA) \subseteq A\\
		&\Leftrightarrow uA \subseteq A^{[p^{e}]}
	\end{align*}
\end{proof}

Next we recall Fedder's Lemma which translates the problem of finding compatible ideals of $R/I$ for an ideal $I$ to finding compatible ideals on $R$. In the case that $R$ is a Gorenstein local ring, this lemma was proved by R. Fedder in \cite{RF}.

\begin{lemma}\cite[Lemma 1.6]{RF}\cite[Lemma 6.2.1]{BKS} \label{fedder} Let $S=R/I$ for some ideal $I$ and $\pi_{e}$ the trace map, then for any $\phi \in \homm_{R}(F_{*}^{e}R,R)$ satisfies $\phi(F_{*}^{e}I)\subseteq I$ if and only if there exists an element $u\in (I^{[p^{e}]}:I)$ such that $\phi(-)=\pi_{e}(F_{*}u-)$. More generally, there exists an isomorphism of $F_{*}^{e}S$-modules
	\[
	\homm_{S}(F_{*}^{e}S,S)\cong \dfrac{\big(F_{*}^{e}(I^{[p^{e}]}:I)\big)}{\big(F_{*}^{e}I^{[p^{e}]}\big)}.
	\]
\end{lemma}

\begin{proof}
	By Lemma \ref{trace}, for any $\phi \in \homm_{R}(F_{*}^{e}R,R)$ there exists an element $u\in R$ such that $\phi(-)=\pi_{e}(F_{*}u-)$. Then by Lemma \ref{fa1},
	\[
	\phi(F_{*}^{e}I)=\pi_{e}(F_{*}^{e}uI)\subseteq I \Leftrightarrow uI \subseteq I^{[p^{e}]} \Leftrightarrow u \in (I^{[p^{e}]}:I).
	\]
	For the second claim, we shall show that the map $\Phi:F_{*}^{e}(I^{[p^{e}]}:I)\rightarrow\homm_{S}(F_{*}^{e}S,S)$ which sends $F_{*}^{e}z$ to the map $\pi_{e}(F_{*}^{e}z-)$ is surjective. It is easy to verify that this map is well-defined and $F_{*}^{e}R$-linear. Since $\homm_{R}(F_{*}^{e}S,S)=\homm_{S}(F_{*}^{e}S,S)$, by freeness of $F_{*}^{e}R$, for any map $\varphi \in \homm_{S}(F_{*}^{e}S,S)$ there always exists a map $\psi \in\homm_{R}(F_{*}^{e}R,R)$ such that $I$ is $\psi$-compatible. Namely, $\Phi$ is surjective. On the other hand, by Lemma \ref{fa1} again, $\kerr\Phi=(F_{*}^{e}I^{[p^{e}]})$, and the result follows by the first isomorphism theorem.
\end{proof}

\begin{lemma}\cite[Proposition 2.6.c]{K2}
	If $\phi$ is surjective, then the set of $\phi$-compatible ideals is a finite set of radicals closed under sum and primary decomposition.
\end{lemma}

For $\phi$-compatible prime ideals $P\subsetneq Q$, we say that $Q$ minimally contains $P$ if there is no $\phi$-compatible prime ideal strictly between $P$ and $Q$. For a given $\phi$-compatible prime ideal $P$, next proposition shows that how to compute $\phi$-compatible prime ideals which minimally contain $P$, and we turn it into an algorithm (cf. \cite[Theorem 4.1]{K3} and \cite[Section 4]{K2}).

\begin{proposition} \label{katzsch}
	Let $\phi:F_{*}^{e}R \rightarrow R$ be an $R$-linear map where $\phi(-)=\pi_{e}(F_{*}^{e}u-)$ for some $u\in R$. Let $P$ and $Q$ be $\phi$-compatible prime ideals such that $Q$ minimally contains $P$, and let $J$ be the ideal whose image in $R/P$ defines the singular locus of $R/P$. Then:
	\begin{enumerate}
		\item If $(P^{[p^{e}]}:P) \subseteq (Q^{[p^{e}]}:Q)$ then $J \subseteq Q$,
		\item If $(P^{[p^{e}]}:P) \nsubseteq (Q^{[p^{e}]}:Q)$ then $(uR+P^{[p^{e}]}):(P^{[p^{e}]}:P) \subseteq Q$.
	\end{enumerate} 
\end{proposition}

\begin{proof}
	For (\textit{1}), let $R_{Q}$ be the localization of $R$ at $Q$, and let $S=\widehat{R_{Q}}$ be the completion of $R_{Q}$ with respect to the maximal ideal $QR_{Q}$. Since colon ideals, Frobenius powers and singular locus commute with localization and completion,
	\begin{align*}
	P \text{ and } Q \text{ are $\phi$-compatible} \Rightarrow & uP\subseteq P^{[p^{e}]} \text{ and } uQ \subseteq Q^{[p^{e}]}\\ \Rightarrow & uPS \subseteq PS^{[p^{e}]} \text{ and } uQS \subseteq QS^{[p^{e}]}\\ \Rightarrow & PS \text{ and } QS \text{ are $u$-special ideals of } S\\ &\text{with } QS \text{ is a prime ideal}
	\end{align*}
	and we have $(P^{[p^{e}]}:P)\subseteq (Q^{[p^{e}]}:Q) \Rightarrow (PS^{[p^{e}]}:PS)\subseteq (QS^{[p^{e}]}:QS)$. 
	Thus, by \cite[Theorem 4.1]{K3}, $JS \subseteq QS$, and  so $J \subseteq Q$. For (\textit{2}), we refer to \cite[Theorem 4.1]{K3}.	
\end{proof}

The following algorithm is the same algorithm described in \cite{K2}, which we call it here the Katzman-Schwede algorithm, finds all $\phi$-compatible prime ideals of $R$ which do not contain $\Ie_{e}(uR)$. We describe it here in a more algebraic language.

\subsection*{Input:}
An $R$ linear map $\phi:F_{*}R\rightarrow R$ where $\phi(-)=\pi_{e}(F_{*}^{e}u-)$ and $u\in R$.
\subsection*{Output:}
Set of all $\phi$-compatible prime ideals which do not contain $\Ie_{e}(uR)$.
\subsection*{Initialize:}
$\mathcal{A}_{R}=\{0\}$ and $\mathcal{B}=\emptyset$
\subsection*{Execute the following:}
While $\mathcal{A}_{R} \neq \mathcal{B}$ pick any $P \in \mathcal{A}_{R}-\mathcal{B}$, set $S=R/P$;
\begin{enumerate}
	\item Find the ideal $J \subseteq R$ whose image in $S$ defines the singular locus of $S$, and compute $J^{\star^{e} u}$,
	\item Find the minimal prime ideals of $J^{\star^{e} u}$, add them to $\mathcal{A}_{R}$,
	\item Compute the ideal $B:=((uR + P^{[p^{e}]}):(P^{[p^{e}]}:P))$, and compute $B^{\star^{e}u}$,
	\item Find the minimal prime ideals of $B^{\star^{e}u}$, add them to $\mathcal{A}_{R}$,
	\item Add $P$ to $\mathcal{B}$.
\end{enumerate}
Output $\mathcal{A}_{R}$ and stop.\\

The Katzman-Schwede algorithm produces a list of all $\phi$-compatible prime ideals which do not contain $L:=I_{e}(uR)$. Because for any prime ideal $Q$, whenever $L \subseteq Q$ we have the property that $Q$ is $\phi$-compatible if and only if $Q/L$ is $\phi'$-compatible where $\phi'$ is the induced map from $\phi$. But $Q/L$ is clearly compatible since $\phi'$ is zero. Thus, we do not need to assume that $\phi$ is surjective.

\begin{discussion}\label{diss}
	Let $R_{\mathfrak{p}}$ be a localization of $R$ at a prime ideal $\mathfrak{p}$, and let $\widehat{R_{\mathfrak{p}}}$ be the completion of $R_{\mathfrak{p}}$ with respect to the maximal ideal $\mathfrak{p}R_{\mathfrak{p}}$. We know that $\mathfrak{p}\widehat{R_{\mathfrak{p}}}$ is the maximal ideal of $\widehat{R_{\mathfrak{p}}}$. Now let $X_{1},\dots,X_{s}$ be minimal generators of $\mathfrak{p}\widehat{R_{\mathfrak{p}}}$, and let $\field{K}[\![X_{1},\dots,X_{s}]\!]$ be the formal power series ring over the residue field $\field{K}$ of $R_{\mathfrak{p}}$. By the Cohen's structure theorem $S\cong \field{K}[\![X_{1},\dots,X_{s}]\!]$. Let $E=E_{S}(S/\mathfrak{m})$ be the injective hull of the residue field. Then by \cite[13.5.3 Example]{RY}, $E$ is isomorphic to the module of inverse polynomials $\field{K}[X_{1}^{-},\dots ,X_{s}^{-}]$ whose $R$-module structure is extended from the following rule
	\begin{align*}
	(\lambda X_{1}^{\alpha_{1}}\dots X_{s}^{\alpha_{s}})&(\mu X_{1}^{-\nu_{1}}\dots X_{s}^{-\nu_{s}})\\&=
	\begin{cases}
	\lambda\mu X_{1}^{-\nu_{1}+\alpha_{1}}\dots X_{s}^{-\nu_{s}+\alpha_{s}} & \text{ if } \alpha_{i} < \nu_{i} \text{ for all } i \\
	0 & \text{ if } \alpha_{i} \geq \nu_{i} \text{ for any } i 
	\end{cases}
	\end{align*}
	for all $\lambda , \mu \in \Bbbk$, non-negative integers $\alpha_{1},\dots,\alpha_{s}$, and positive integers $\nu_{1},\dots,\nu_{s}$. Therefore, $E$ has a natural $S[T;f^{e}]$-module structure by extending additively the action $T(\lambda X_{1}^{\alpha_{1}}\dots X_{s}^{\alpha_{s}})=\lambda^{p^{e}} X_{1}^{-p^{e}\nu_{1}}\dots X_{s}^{-p^{e}\nu_{s}}$ for all $\lambda \in \field{K}$ and positive integers $\nu_{1},\dots,\nu_{s}$. Notice that $T:E\rightarrow E$ defines a Frobenius map. 
	
	Following \cite[Section 4]{K3}, we can also view the Katzman-Schwede algorithm from the point of Frobenius maps on injective hull of residue fields. Any $S[\Theta;f^{e}]$-module structure on $E$ can be given by $\Theta=uT$ for some $u\in S$ where $T$ is the natural action as above. We also know that the set of $S$-submodules of $E$ is $\{\ann_{E}J \mid J \text{ is an ideal of } R \}$. In addition, \cite[Theorem 4.3]{K1} shows that an $S$-submodule $\ann_{E}J \subseteq E$ is an $S[\Theta ; f^{e}]$-submodule if and only if $uJ \subseteq J^{[p^{e}]}$. Thus, the Katzman-Schwede algorithm finds all submodules $\ann_{E}P$ of $E$ which are preserved by the Frobenius map $\Theta$, under the assumptions that $P$ is a prime ideal of $S$ and the restriction of $\Theta$ to $\ann_{E}P$ is not the zero map (i.e. it finds all the $\Theta$-special prime ideals of $S$, see Definition \ref{uspecial}).
\end{discussion}

All of the operations used in the Katzman-Schwede algorithm are defined for localizations of $R$. Therefore, we can apply the algorithm to any localization of $R$ at a prime ideal. In the rest of this section, we investigate behaviour of the Katzman-Schwede algorithm under localization. Let $R_{\mathfrak{p}}$ be a localization of $R$ at a prime ideal $\mathfrak{p}$. Our next theorem gives the exact relation between the output sets $\mathcal{A}_{R}$ and $\mathcal{A}_{R_{\mathfrak{p}}}$ of the Katzman-Schwede algorithm for $R$ and $R_{\mathfrak{p}}$, respectively.

\begin{theorem}\label{al1} The Katzman-Schwede algorithm commutes with localization: for a given $u \in R$, if $\mathcal{A}_{R}$ and $\mathcal{A}_{R_{\mathfrak{p}}}$ are the output sets of the Katzman-Schwede algorithm for $R$ and $R_{\mathfrak{p}}$, respectively, then
	\[
	\mathcal{A}_{R_{\mathfrak{p}}}=\{ PR_{\mathfrak{p}} \mid P \in \mathcal{A}_{R} \text{ and } P \subseteq \mathfrak{p} \}
	\]
\end{theorem}

\begin{proof}
	We shall show that the Katzman-Schwede algorithm commutes with localization step by step. Since the ideal defining singular locus commutes with localization, so is step \textit{1.} Since Frobenius powers and colon ideals commute with localization under Noetherian hypothesis, so is step \textit{3.} Then by Lemma \ref{sp2}, $\star$-closure commutes with localization. Therefore, step \textit{2.} and \textit{4.} follow from the fact that primary decomposition commutes with localization.
	
	Let $P$ be a $\phi$-compatible prime ideal of $R$. Then since $uP\subseteq P^{[p^{e}]} \Leftrightarrow uPR_{\mathfrak{p}}\subseteq P^{[p^{e}]}R_{\mathfrak{p}}$, $PR_{\mathfrak{p}}$ is a $\phi$-compatible prime ideal of $R_{\mathfrak{p}}$. Since the Katzman-Schwede algorithm commutes with localization, $Q$ is a $\phi$-compatible prime ideal of $R$ minimally containing $P$ if and only if $QR_{\mathfrak{p}}$ is a $\phi$-compatible prime ideal of $R_{\mathfrak{p}}$ minimally containing $PR_{\mathfrak{p}}$. Hence, $\mathcal{A}_{R_{\mathfrak{p}}}=\{ PR_{\mathfrak{p}} \mid P \in \mathcal{A}_{R} \text{ and } P \subseteq \mathfrak{p} \}$.
\end{proof}

\section{A Generalization of the Katzman-Zhang Algorithm}
\label{section:Katzman-Zhang Algorithm}
Let $R=\Bbbk[x_{1},\dots,x_{n}]$ be a polynomial ring over a field of characteristic $p$ and $e$ be a positive integer. Let $R_{\mathfrak{p}}$ be a localization of $R$ at a prime ideal $\mathfrak{p}$, and let $S=\widehat{R_{\mathfrak{p}}}$ be the completion of $R_{\mathfrak{p}}$ with respect to the maximal ideal $\mathfrak{m}=\mathfrak{p}R_{\mathfrak{p}}$. Let $E=E_{S}(S/\mathfrak{m})$ be the injective hull of residue field of $S$. The purpose of this section is to generalize the algorithm defined in Section 6 of \cite{K3} to $R$, and show that it commutes with localization.

\begin{remark}
	Given an Artininan $S$-module $M$, we can embed $M$ in $E^{\alpha}$ for some positive integer $\alpha$, we can then embed $\coker(M\hookrightarrow E^{\alpha})$ in $E^{\beta}$ for some positive integer $\beta$. Continuing in this way, we get an injective resolution
	\[
	0\rightarrow M\rightarrow E^{\alpha}\xrightarrow{A^{t}}E^{\beta}\rightarrow \cdots
	\]
	of $M$, where $A$ is an $\alpha\times\beta$ matrix with entries in $S$ since $\homm_{S}(E^{\alpha},E^{\beta})\cong\homm_{S}(S^{\alpha},S^{\beta})$, and so $M\cong \ker A^{t}$.
\end{remark}

\begin{remark} \label{skew}
	Let  $T: E\rightarrow E$ be as in Discussion \ref{diss}. We can extend this natural $S[T;f^{e}]$-module structure on $E$ to $E^{\alpha}$ which is given by 
	\[
	T \left( \begin{array}{ccc}
	a_{1}  \\
	\vdots  \\
	a_{\alpha} \end{array} \right) = 
	\left( \begin{array}{ccc}
	Ta_{1}  \\
	\vdots  \\
	Ta_{\alpha} \end{array} \right)
	\]
	for all $a_{1},\dots,a_{\alpha}\in E$.
\end{remark}

\begin{remark}
Following \cite[Section 3]{K1}, let $\mathfrak{C}^{e}$ be the category of Artinian $S[\theta;f^{e}]$-modules and $\mathfrak{D}^{e}$ be the category of $S$-linear maps $M\rightarrow F_{S}^{e}(M)$ where $M$ is Noetherian $S$-module and a morphism between $M\rightarrow F_{S}^{e}(M)$ and $N\rightarrow F_{S}^{e}(N)$ is a commutative diagram of $S$-linear maps
$$\begin{array}[c]{ccc}
M&\stackrel{\phi}{\longrightarrow}&N\\
\downarrow\scriptstyle{}&&\downarrow\scriptstyle{}\\
F_{S}^{e}(M)&\stackrel{F_{S}^{e}(\phi)}{\longrightarrow}&F_{S}^{e}(N)
\end{array}$$
We define the functor $\Delta^{e}:\mathfrak{C}^{e}\rightarrow \mathfrak{D}^{e}$ as follows: given an $e$-th Frobenius map $\theta:M\rightarrow M$, we can obtain an $R$-linear map $\phi:F_{*}^{e}R\otimes M\rightarrow M$ such that $\phi(F_{*}r\otimes m)=r\theta(m)$ for all $r\in R$, $m\in M$. Applying Matlis duality to this map gives the $R$-linear map $M^{\vee}\rightarrow (F_{*}^{e}R\otimes M)^{\vee}\cong F_{*}^{e}R\otimes M^{\vee}$ where the last isomorphism is described in \cite[Lemma 4.1]{L1}. Conversely, we define the functor $\Psi^{e}:\mathfrak{D}^{e}\rightarrow \mathfrak{C}^{e}$ as follows: given a Noetherian $R$-module $N$ with an $R$-linear map $N\rightarrow F_{R}^{e}(N)$. Applying Matlis duality to this map gives the $R$-linear map $\varphi:F_{R}^{e}(N^{\vee})\cong F_{R}^{e}(N)^{\vee}\rightarrow N^{\vee}$ where the first isomorphism is the composition $F_{R}^{e}(N^{\vee})\cong F_{R}^{e}(N^{\vee})^{\vee\vee}\cong F_{R}^{e}(N^{\vee\vee})^{\vee}\cong F_{R}^{e}(N)^{\vee}$. Then we define the action of $\theta$ on $N^{\vee}$ by defining $\theta(n)=\varphi(1\otimes n)$ for all $n\in N^{\vee}$.
\end{remark}

The mutually inverse exact functors $\Delta^{e}$ and $\Psi^{e}$ are extensions of Matlis duality which also keep track of Frobenius actions.

\begin{proposition}\label{art1}\cite[Proposition 2.1]{K3}
	Let $M\cong\ker A^{t}$ be an Artininan $S$-module where $A$ is an $\alpha\times\beta$ matrix with entries in $S$. For a given $e$-th Frobenius map on $M$, $\Delta^{e}(M)\in\homm_{S}(\coker A,\coker A^{[p^{e}]})$ and is given by an $\alpha\times\alpha$ matrix $U$ such that $U\im A\subseteq\im A^{[p^{e}]}$, conversely any such $U$ defines an $S[\Theta;f^{e}]$-module structure on $M$ which is given by the restriction to $M$ of the Frobenius map $\Theta:E^{\alpha} \rightarrow E^{\alpha}$ defined by $\Theta(a)=U^{t}T(a)$ for all $a\in E^{\alpha}$.
\end{proposition}

\begin{remark}\label{rm1}
	By Proposition \ref{art1}, for any Artinian submodule $M\cong\ker A^{t}$ of $E^{\alpha}$ with a given $S[\Theta;f^{e}]$-module structure, where $\Theta=U^{t}T$, there is a submodule $V$ of $S^{\alpha}$ such that $M=\ann_{E^{\alpha}}V^{t}:=\{a\in E^{\alpha} \mid V^{t}a=0 \}$ and $UV \subseteq V^{[p^{e}]}$, (in fact $V=\im A$). For simplicity, for $V \subseteq S^{\alpha}$ we denote $E(V)=\ann_{E^{\alpha}}V^{t}$.
\end{remark}

\begin{lemma} \label{nil1}\cite[Lemma 3.6, Lemma 3.7]{K3}
	Let $\Theta=U^{t}T:E^{\alpha}\rightarrow E^{\alpha}$ be a Frobenius map where $U$ is an $\alpha\times\alpha$ matrix with entries in $S$ and let $K\subset S^{\alpha}$. Then
	\begin{enumerate}
		\item $E(I_{e}(\im U^{[p^{e}-1]}U^{[p^{e}-2]}\cdots U))=\{a\in E^{\alpha} \mid \Theta^{e}(a)=0 \}$,
		\item $E(I_{1}(UK))=\{ a\in E^{\alpha} \mid \Theta(a)\in E(K) \}$.
	\end{enumerate}
\end{lemma}

\begin{remark}
	Let $M=\ann_{E^{\alpha}}V^{t}$ be as in Remark \ref{rm1}. Then $\ann_{S}M=\ann_{S}S^{\alpha}/V$ because $\ann_{S}M \subseteq \ann_{S}M^{\vee} \subseteq \ann_{S}M^{\vee\vee}\cong\ann_{S}M$.
\end{remark}

\begin{definition} \label{uspecial}
	Let $\Theta=U^{t}T: E^{\alpha} \rightarrow E^{\alpha}$ be a Frobenius map, where $U$ is an $\alpha \times \alpha$ matrix with entries in $S$. We call an ideal of $S$ a $\Theta$-special ideal if it is an annihilator of an $S[\Theta;f]$-submodule of $E^{\alpha}$, equivalently if it is the annihilator of $S^{\alpha}/W$ for some $W\subset S^{\alpha}$ with $UW \subseteq W^{[p^{e}]}$.
\end{definition}

Notice that the concept of injective hull of the residue field is not available for polynomial rings. Therefore, we adapt above definition for a more general setting and define special ideals depending on a given square matrix as follows.

\begin{definition}
	Let $\mathcal{R}$ be $R$ or $R_{\mathfrak{p}}$ or $S$. For a given $\alpha\times\alpha$ matrix $U$ with entries in $\mathcal{R}$, we call an ideal of $\mathcal{R}$ a $U$-special ideal if it is the annihilator of $\mathcal{R}^{\alpha}/V$ for some submodule $V\subseteq \mathcal{R}^{\alpha}$  satisfying $UV\subseteq V^{[p^{e}]}$.
\end{definition}

Next we will provide some properties of special ideals. The following lemma gives the most important properties which are actually generalization of Lemma 3.8 and 3.10 in \cite{K3} to $R$ with similar proofs.

\begin{lemma}\label{sp1}
	Let $\mathcal{R}$ be $R$ or $R_{\mathfrak{p}}$ or $S$. Let $U$ be an $\alpha\times\alpha$ matrix with entries in $\mathcal{R}$ and $J$ be a $U$-special ideal of $\mathcal{R}$. Then
	\begin{enumerate}
		\item Associated primes of $J$ are $U$-special,
		\item $V=(JR^{\alpha})^{\star^{e} U}$ is the smallest submodule of $\mathcal{R}^{\alpha}$ such that $J=\ann_{\mathcal{R}}\mathcal{R}^{\alpha}/V$ and $UV\subseteq V^{[p^{e}]}$.
	\end{enumerate}
\end{lemma}

\begin{proof}
	For \textit{1.} let $P$ be an associated prime of $J$ and $J=\ann_{\mathcal{R}}\mathcal{R}^{\alpha}/V$ for some $V\subseteq \mathcal{R}^{\alpha}$ such that $UV\subseteq V^{[p^{e}]}$. Then for a suitable element $r\in \mathcal{R}$ we have $P=(J: r)$. If $W=(V:_{\mathcal{R}^{\alpha}} r)=\{ w\in \mathcal{R}^{\alpha}\mid rw\in V \}$ then $P=\ann_{\mathcal{R}}\mathcal{R}^{\alpha}/W$ since $s \in P \Leftrightarrow rs \in J \Leftrightarrow rs\mathcal{R}^{\alpha} \subseteq V \Leftrightarrow s\mathcal{R}^{\alpha} \subseteq W$. On the other hand, since $UV \subseteq V^{[p^{e}]}$ and $rW \subseteq V$ we have $rUW \subseteq UV$ and so $r^{p^{e}}UW\subseteq r^{p^{e}-1}UV\subseteq r^{p^{e}-1}V^{[p^{e}]} \subseteq V^{[p^{e}]}$. This means that $UW \subseteq (V^{[p^{e}]}:_{\mathcal{R}^{\alpha}} r^{p^{e}})=(V:_{\mathcal{R}^{\alpha}} r)^{[p^{e}]}=W^{[p^{e}]}$.
	
	For \textit{2.} let $J=\ann_{\mathcal{R}}\mathcal{R}^{\alpha}/V$ for some $V\subseteq \mathcal{R}^{\alpha}$ such that $UV\subseteq V^{[p^{e}]}$. It is clear that $J\mathcal{R}^{\alpha}\subseteq (J\mathcal{R}^{\alpha})^{\star^{e} U}$ and $J\mathcal{R}^{\alpha} \subseteq V \Rightarrow (J\mathcal{R}^{\alpha})^{\star^{e} U}\subseteq V^{\star^{e} U}=V$. Therefore, $J \subseteq \ann_{\mathcal{R}}\mathcal{R}^{\alpha}/(J\mathcal{R}^{\alpha})^{\star^{e} U} \subseteq \ann_{\mathcal{R}}\mathcal{R}^{\alpha}/V=J$, and so $J=\ann_{\mathcal{R}}\mathcal{R}/(J\mathcal{R}^{\alpha})^{\star^{e} U}$.
\end{proof}

\begin{theorem}\label{T1}\cite[Theorem 5.1]{K3} There are only finitely many $\Theta$-special prime ideals $P$ of $S$ with the property that for some $S[\Theta;f]$-submodule $M \subseteq E^{\alpha}$ with $\ann_{S}M=P$ and the restriction of $\Theta$ to $M$ is not zero.	
\end{theorem}

Theorem \ref{T1} was proved by induction on $\alpha$ using the aid of injective hull of the residue field of $S$, and turned into an algorithm in \cite{K3}, which we call it here Katzman-Zhang Algorithm. Since injective hulls of residue fields are not available for polynomial rings, we only use techniques of $\Ie_{e}(-)$ operation and $\star$-closure to generalize the Katzman-Zhang Algorithm to $R$. Next theorem allows us to prove polynomial version of Theorem \ref{T1}.

\begin{theorem} \label{2} \cite[Theorem 3.2]{K4}
	Let $U$ be an $\alpha \times \alpha$ matrix with entries in $R$ and $\alpha\in\N$.
	\begin{enumerate}
		\item If $\Ie_{e}(U^{[p^{e-1}]}U^{[p^{e-2}]} \cdots UR^{\alpha})=\Ie_{e+1}(U^{[p^{e}]}U^{[p^{e-1}]} \cdots UR^{\alpha})$ then
		\[
		\Ie_{e}(U^{[p^{e-1}]}U^{[p^{e-2}]} \cdots UR^{\alpha})= \Ie_{e+j}(U^{[p^{e+j-1}]}U^{[p^{e+j-2}]} \cdots UR^{\alpha}) 
		\] 
		for all $j \geq 0$.
		\item There exists an integer $e$ such that (1) holds.
	\end{enumerate}
\end{theorem}

For the rest of this section, we will fix an $\alpha\times\alpha$ matrix $U$ with entries in $R$, and $\mathcal{K}$ will denote the stable value of $\{\Ie_{e}(U^{[p^{e-1}]}U^{[p^{e-2}]} \cdots UR^{\alpha})\}_{e\geq1}$ as in Theorem \ref{2}.

\begin{proposition} \label{rmklast}
	If $P$ is a prime ideal of $R$ with the property that $\mathcal{K} \subseteq PR^{\alpha}$ where $\mathcal{K}=\Ie_{e}(U_{e}R^{\alpha})$ and $U_{e}=U^{[p^{e-1}]}U^{[p^{e-2}]}\cdots U$, then $P$ is $U_{e}$-special.	
\end{proposition} 

\begin{proof} 
	Let $P$ be a prime ideal of $R$ such that $\mathcal{K} \subseteq PR^{\alpha}$. Then
	\[
	\mathcal{K} \subseteq PR^{\alpha}\Rightarrow U_{e}R^{\alpha}\subseteq P^{[p^{e}]}R^{\alpha}\Rightarrow U_{e}PR^{\alpha}\subseteq P^{[p^{e}]}R^{\alpha}\Rightarrow PR^{\alpha}=(PR^{\alpha})^{\star U_{e}}.
	\]
	Therefore, $P$ is $U_{e}$-special.
\end{proof}

By Proposition \ref{rmklast}, any prime ideal containing $\mathcal{K}$ is $U_{e}$-special. This is equivalent to saying that the action of $U_{e}$ on submodules $PR^{\alpha}$ containing $\mathcal{K}$ with $P$ being a prime is the same as the action of zero matrix. Henceforth, we will assume that $\mathcal{K} \neq 0$.

Our next theorem is the generalization of Theorem \ref{2} to $R$, and we will prove it using a very similar method to that in \cite[Section 5]{K3}. 

\begin{theorem} \label{1}
	The set of all $U$-special prime ideals $P$ of $R$ with the property that $\mathcal{K} \nsubseteq PR^{\alpha}$ is finite. 
\end{theorem}

We will prove Theorem \ref{1} by induction on $\alpha$. Assume that $\alpha=1$. For a prime ideal $P$ being a $u$-special prime, i.e. $P=\ann_{R}R/P^{\star u}$, is equivalent to the property that $uP \subseteq P^{[p]}$. This means, by Corollary \ref{fac}, that $P$ is a $\phi$-compatible ideal where $\phi(-)=\pi(F_{*}u-)$. Then the set of all $u$-special prime ideals are finite and the Katzman-Schwede algorithm finds such primes. Henceforth in this section, we will assume that Theorem \ref{1} holds for $\alpha-1$.

For a $U$-special prime ideal $P$, we will present an effective method for finding all $U$-special prime ideals $Q\varsupsetneq P$ for which there is no $U$-special prime ideal strictly between $P$ and $Q$, and we will call such $U$-special prime ideals $Q$ as minimally containing $P$. The following lemma is a generalization of Lemma 5.2 in \cite{K3} to $R$, which is our starting point of finding $U$-special prime ideals minimally containing $P$.

\begin{lemma} \label{last1}
	Let $P\subsetneq Q$ be $U$-special prime ideals of $R$ such that $Q$ contains $P$ minimally. If $a\in Q\setminus P$, then $Q$ is among the minimal prime ideals of $\ann_{R}R^{\alpha}/W$ where $W=((P+aR)R^{\alpha}) ^{\star U}$.
\end{lemma}

\begin{proof}
	Since $PR^{\alpha}\subseteq (P+aR)R^{\alpha} \subseteq QR^{\alpha}$ we have
	\[
	(PR^{\alpha})^{\star U}\subseteq ((P+aR)R^{\alpha})^{\star U} \subseteq (QR^{\alpha})^{\star U}.
	\]
	Then by Lemma \ref{sp1},
	\[
	P=\ann_{R}\frac{R^{\alpha}}{(PR^{\alpha})^{\star U}} \subseteq \ann_{R}\frac{R^{\alpha}}{W} \subseteq \ann_{R}\frac{R^{\alpha}}{(QR^{\alpha})^{\star U}}=Q
	\]
	which implies that $Q$ contains a minimal prime ideal of $\ann_{R}R^{\alpha}/W$. Therefore, by Lemma \ref{sp1} again, this minimal prime is $U$-special. Since $Q$ contains $P$ minimally, it has to be $Q$ itself.  
\end{proof}

Next, we will prove a generalization of Lemma 5.3 in \cite{K3} to $R$, which is a crucial step for proving Theorem \ref{1}.

\begin{lemma} \label{last2}
	Let $Q$ be a $U$-special prime ideal of $R$, where $Q=\ann_{R}R^{\alpha}/W$ for some submodule $W\subseteq R^{\alpha}$ satisfying $UW\subseteq W^{[p]}$. Let $a\notin Q$ and $X$ be an invertible $\alpha\times\alpha$ matrix with entries in the localization $R_{a}$. Let $\nu \gg 0$ be such that $U_{1}=a^{\nu}X^{[p]}UX^{-1}$ has entries in $R$ and $W_{1}=XW_{a}\cap R^{\alpha}$. Then
	\begin{enumerate}
		\item $Q$ is a minimal prime of $\ann_{R}R^{\alpha}/W_{1}$ and $U_{1}W_{1}\subseteq W_{1}^{[p]}$, i.e. $Q$ is $U_{1}$-special.
		\item If $\Ie_{e}(U^{[p^{e-1}]}U^{[p^{e-2}]} \cdots UR^{\alpha})\nsubseteq W$, then $$\Ie_{e}(U_{1}^{[p^{e-1}]}U_{1}^{[p^{e-2}]} \cdots U_{1}R^{\alpha})\nsubseteq W_{1}.$$
	\end{enumerate}
\end{lemma}

\begin{proof}
	Let $J=\ann_{R}R^{\alpha}/W_{1}$. Then 
	\begin{align*}
		J_{a}
		&=(\ann_{R}R^{\alpha}/W_{1})_{a}=\ann_{R_{a}}R_{a}^{\alpha}/(W_{1})_{a}=\ann_{R_{a}}R_{a}^{\alpha}/XW_{a}\\
		&\cong \ann_{R_{a}}R_{a}^{\alpha}/W_{a}=(\ann_{R}R^{\alpha}/W)_{a}=Q_{a}.
	\end{align*}
	Therefore, $Q$ is a minimal prime ideal of $J$. We also have
	\begin{align*}
		U_{1}W_{1}&=a^{\nu}X^{[p]}UX^{-1}(XW_{a}\cap R^{\alpha})\subseteq(a^{\nu}X^{[p]}UX^{-1}XW_{a})\cap R^{\alpha}\\
		&\subseteq X^{[p]}W_{a}^{[p]}\cap R^{\alpha}=(XW_{a})^{[p]} \cap R^{\alpha}=(XW_{a}\cap R^{\alpha})^{[p]}=W_{1}^{[p]}.
	\end{align*}
	This means that $J$ is $U_{1}$-special. Therefore, by Lemma \ref{sp1}, $Q$ is $U_{1}$-special.
	
	Assume that
	\[
	\Ie_{e}(U^{[p^{e-1}]}U^{[p^{e-2}]} \cdots UR^{\alpha})\nsubseteq W \text{, i.e. } U^{[p^{e-1}]}U^{[p^{e-2}]} \cdots UR^{\alpha}\nsubseteq W^{[p^{e}]}.
	\]
	Now suppose the contrary that
	\[
	\Ie_{e}(U_{1}^{[p^{e-1}]}U_{1}^{[p^{e-2}]} \cdots U_{1}R^{\alpha})\subseteq W_{1}\text{, i.e. } U_{1}^{[p^{e-1}]}U_{1}^{[p^{e-2}]} \cdots U_{1}R^{\alpha}\subseteq W_{1}^{[p^{e}]}.
	\]
	Since
	\begin{align*}
		&U_{1}^{[p^{e-1}]}U_{1}^{[p^{e-2}]} \cdots U_{1}=(a^{\nu}X^{[p]}UX^{-1})^{[p^{e-1}]}(a^{\nu}X^{[p]}UX^{-1})^{[p^{e-2}]} \cdots a^{\nu}X^{[p]}UX^{-1}\\
		&=a^{\nu(p^{e-1})}X^{[p^{e}]}U^{[p^{e-1}]}(X^{-1})^{[p^{e-1}]}a^{\nu(p^{e-2})}X^{[p^{e-1}]}U^{[p^{e-2}]}(X^{-1})^{[p^{e-2}]}\cdots a^{\nu}X^{[p]}UX^{-1}\\
		&=a^{\nu(p^{e-1}+p^{e-2}+\cdots+1)}X^{[p^{e}]}U^{[p^{e-1}]}U^{[p^{e-2}]}\cdots UX^{-1},
	\end{align*}
	we have $bX^{[p^{e}]}U^{[p^{e-1}]}U^{[p^{e-2}]}\cdots UX^{-1}R^{\alpha}  \subseteq W_{1}^{[p^{e}]}=(XW_{a}\cap R^{\alpha})^{[p^{e}]}=X^{[p^{e}]}W_{a}^{^{[p^{e}]}}\cap R^{\alpha}$, where $b=a^{\nu(p^{e-1}+p^{e-2}+\cdots+1)}$. Therefore, $$X^{[p^{e}]}U^{[p^{e-1}]}U^{[p^{e-2}]}\cdots UX^{-1}R_{a}^{\alpha}  \subseteq X^{[p^{e}]}W_{a}^{^{[p^{e}]}},$$ and so $U^{[p^{e-1}]}U^{[p^{e-2}]}\cdots UR_{a}^{\alpha}  \subseteq W_{a}^{^{[p^{e}]}}$. Then $U^{[p^{e-1}]}U^{[p^{e-2}]}\cdots UR^{\alpha}  \subseteq W^{^{[p^{e}]}}$ since $a$ is not a zero divisor on $R^{\alpha}/W^{[p^{e}]}$, which contradicts with our assumption.
\end{proof}

Next, we will give a generalization of Proposition 5.4 in \cite{K3} to $R$, which will give us an effective method for finding the $U$-special prime ideals containing a $U$-special prime $P$ minimally in an important case.

\begin{proposition} \label{last3}
	Let $P$ be a $U$-special prime ideal of $R$ such that $\mathcal{K} \nsubseteq PR^{\alpha}$. Assume that the $\alpha$-th column of $U$ is zero and $PR^{\alpha}=(PR^{\alpha})^{\star U}$. Then the set of $U$-special prime ideals minimally containing $P$ is finite.
\end{proposition}

\begin{proof}
	Let $Q$ be a $U$-special prime ideal minimally containing $P$ and $W=(QR^{\alpha})^{\star U}$. Let $U_{0}$ be the top left $(\alpha-1)\times(\alpha-1)$ submatrix of $U$. Since $PR^{\alpha}=(PR^{\alpha})^{\star U}\Leftrightarrow UPR^{\alpha}\subseteq P^{[p]}R^{\alpha}$, all entries of $U$ are in $(P:P^{[p]})$. Therefore, $U_{0}PR^{\alpha-1}\subseteq P^{[p]}R^{\alpha-1}$, and so $P$ is $U_{0}$-special. Let $\mathcal{K}_{0}$ be the stable value of $\{ \Ie_{e}(U_{0}^{[p^{e-1}]}U_{0}^{[p^{e-2}]}\cdots U_{0}R^{\alpha-1}) \}_{e>0}$ as in Theorem \ref{2}. We now split our proof into two parts. Assume first that $\mathcal{K}_{0}\subseteq PR^{\alpha-1}$, i.e. $\Ie_{e}(U_{0}^{[p^{e-1}]}U_{0}^{[p^{e-2}]}\cdots U_{0}R^{\alpha-1}) \subseteq PR^{\alpha-1}$ for some $e>0$.
	\begin{enumerate}
		\item[1)] Let $(g_{1},\dots,g_{\alpha-1},0)$ be the last row of the matrix $U^{[p^{e-1}]}U^{[p^{e-2}]}\cdots U$. Note that its top left $(\alpha-1)\times(\alpha-1)$ submatrix is $U_{0}^{[p^{e-1}]}U_{0}^{[p^{e-2}]}\cdots U_{0}$. By our assumption, all entries of $U_{0}^{[p^{e-1}]}U_{0}^{[p^{e-2}]}\cdots U_{0}$ are in $P^{[p^{e}]}\subseteq Q^{[p^{e}]}$. Therefore, $\Ie_{e}(U_{0}^{[p^{e-1}]}U_{0}^{[p^{e-2}]}\cdots U_{0}R^{\alpha-1})\subseteq QR^{\alpha-1}$. Then by Proposition \ref{rmklast}, $P$ and $Q$ are $U_{0}^{[p^{e-1}]}U_{0}^{[p^{e-2}]}\cdots U_{0}$-special, and so the action of $U^{[p^{e-1}]}U^{[p^{e-2}]}\cdots U$ is the same action of a matrix $U_{e}$ whose first $\alpha-1$ rows are zero and last row is $(g_{1},\dots,g_{\alpha-1},0)$, and so we replace $U^{[p^{e-1}]}U^{[p^{e-2}]}\cdots U$ with $U_{e}$ without effecting any issues. We now define inductively $V_{0}=QR^{\alpha}$ and $V_{i+1}=\Ie_{e}(U_{e}V_{i})+V_{i}$ for all $i\geq0$. Since
		\[
		U_{e}QR^{\alpha}=\{ (0,\dots,0,\sum_{i=1}^{\alpha-1}g_{i}q_{i})^{t} \mid \forall i, q_{i}\in Q\},
		\]
		\[
		\Ie_{e}(U_{e}QR^{\alpha})=\{ (0,\dots,0,v) \mid v\in\Ie_{e}(\sum_{i=1}^{\alpha-1}g_{i}Q) \}.
		\]
		Therefore, the sequence $\{ V_{i} \}_{i\geq0}$ stabilizes at $V_{1}=\Ie_{e}(U_{e}QR^{\alpha})+QR^{\alpha}$. By definition of $\star$-closure, we have $QR^{\alpha}\subseteq V_{1} \subseteq W$, and so $\ann_{R}R^{\alpha}/V_{1}=Q$. 
		Furthermore, we have
		\[
		\ann_{R}\frac{R}{\Ie_{e}(\sum_{i=1}^{\alpha-1}g_{i}Q)}=\ann_{R}\frac{R^{\alpha}}{\Ie_{e}(U_{e}QR^{\alpha})}\subseteq Q \text{ since } \Ie_{e}(U_{e}QR^{\alpha}) \subseteq V_{1},
		\]
		which implies that
		\[
		\Ie_{e}(\sum_{i=1}^{\alpha-1}g_{i}Q)=\sum_{i=1}^{\alpha-1}\Ie_{e}(g_{i}Q) \subseteq Q,
		\]
		i.e. $\Ie_{e}(g_{i}Q)\subseteq Q \Leftrightarrow g_{i}Q\subseteq Q^{[p^{e}]}$ for all $1\leq i<\alpha$. Hence, $Q$ is $g_{i}$-special for all $1\leq i<\alpha$. On the other hand, at least for one $g_{i}$ we must have $g_{i}\notin P^{[p^{e}]}$ so that we do not get a contradiction with our assumption $\mathcal{K}\nsubseteq PR^{\alpha}$. We can now produce all such $Q$ using the Katzman-Schwede algorithm.
	\end{enumerate}
	
	Let $\tau\subset R$ be intersection of the finite set of $U_{0}$-special prime ideals of $R$ minimally containing $P$. Let $ \rho:R^{\alpha}\rightarrow R^{\alpha-1}$ be the projection onto first $\alpha-1$ coordinates, and let $J=\ann_{R}R^{\alpha-1}/\rho(W)$. Then since $U_{0}\rho(W)=\rho(UW)\subseteq \rho(W^{[p]})=\rho(W)^{[p]}$, $J$ is $U_{0}$-special. Note that $Q\subseteq J$, and so $P\subsetneq J$. Assume now that $\mathcal{K}_{0}\nsubseteq PR^{\alpha-1}$.
	
	\begin{enumerate}
		\item[2)] We now compute $(\tau^{[p^{e}]}\mathcal{K}_{0})^{\star U_{0}}$ as the stable value of
		\begin{align*}
			L_{0}&=\tau^{[p^{e}]}\mathcal{K}_{0}\\
			L_{1}&=\Ie_{1}(U_{0}L_{0})+L_{0}=\tau^{[p^{e-1}]}\Ie_{1}(U_{0}\mathcal{K}_{0})+\tau^{[p^{e}]}\mathcal{K}_{0}=\tau^{[p^{e-1}]}\mathcal{K}_{0}+\tau^{[p^{e}]}\mathcal{K}_{0}\\
			\vdots\\
			L_{e}&=\tau\mathcal{K}_{0}+L_{e-1}\\
			\vdots
		\end{align*}
		and we deduce that $\tau\mathcal{K}_{0}\subseteq L_{e}\subseteq (\tau^{[p^{e}]}\mathcal{K}_{0})^{\star U_{0}}$. On the other hand, since $J$ is a $U_{0}$-special ideal strictly containing $P$, $\tau \subseteq \sqrt{J}$. Thus, for all large $e\geq 0$, we have $\tau^{[p^{e}]} \subseteq J$. Therefore,
		\[
		\tau\mathcal{K}_{0}\subset (\tau^{[p^{e}]}\mathcal{K}_{0})^{\star U_{0}}\subseteq (JR^{\alpha-1})^{\star U}\subseteq \rho(W)^{\star U_{0}}=\rho(W).
		\] 
		where the last equality follows from the fact that $UW\subseteq W^{[p]}$. Moreover, since $\tau \nsubseteq P$, we have $\tau\mathcal{K}_{0} \nsubseteq PR^{\alpha-1}$.
		
		\item[3)] Now we define $\bar{v}=(v_{1},\dots,v_{\alpha-1},0)^{t}$ for $v=(v_{1},\dots,v_{\alpha-1},v_{\alpha})^{t}\in R^{\alpha}$, and $\widebar{V}=\{\bar{v}\mid v\in V\}$ for any submodule $V\subseteq R^{\alpha}$. Let $l:R^{\alpha-1}\rightarrow R^{\alpha-1}\oplus R$ be the natural inclusion $l(v)=v\oplus0$. Note that $\widebar{V}=l(\rho(V))$. Then we also define $W_{0}=\{ w\in W\mid \rho(w)\in\tau\mathcal{K}_{0} \}$ and note that (2) implies that $\rho(W_{0})=\tau\mathcal{K}_{0}$. We have $W_{0}^{\star U} \subseteq W^{\star U}=W$ and $W_{0}^{\star U}=\Ie_{1}(UW_{0})^{\star U} +W_{0}$. Since $UW_{0}=U\widebar{W_{0}}=Ul(\tau\mathcal{K}_{0})$, $\Ie_{1}(Ul(\tau\mathcal{K}_{0}))^{\star U} \subseteq W_{0}^{\star U}\subseteq W$. On the other hand, if $\Ie_{1}(Ul(\tau\mathcal{K}_{0}))^{\star U}\subseteq PR^{\alpha}$, then
		\begin{align*}
			\Ie_{1}(Ul(\tau\mathcal{K}_{0})) \subseteq PR^{\alpha} &\Rightarrow Ul(\tau\mathcal{K}_{0}) \subseteq P^{[p]}R^{\alpha} \Rightarrow \rho(Ul(\tau\mathcal{K}_{0})) \subseteq \rho(P^{[p]}R^{\alpha})\\ &\Rightarrow U_{0}\tau\mathcal{K}_{0} \subseteq P^{[p]}R^{\alpha-1} \Rightarrow \tau^{[p]}U_{0}\mathcal{K}_{0}\subseteq P^{[p]}R^{\alpha-1}\\ &\Rightarrow \Ie_{1}(\tau^{[p]}U_{0}\mathcal{K}_{0})\subseteq PR^{\alpha-1} \Rightarrow \tau\Ie_{1}(U_{0}\mathcal{K}_{0})\subseteq PR^{\alpha-1}\\ &\Rightarrow \tau\mathcal{K}_{0}\subseteq PR^{\alpha-1}
		\end{align*}
		which contradicts with (2). Hence, we also have $\Ie_{1}(Ul(\tau\mathcal{K}_{0}))^{\star U} \nsubseteq PR^{\alpha}$.
		\item[4)] Let $M'$ be a matrix whose columns generate $\Ie_{1}(Ul(\tau\mathcal{K}_{0}))^{\star U}\subseteq W$. Choose an entry $a$ of $M'$ which is not in $P$. Then
		\begin{enumerate}
			\item If $a\in Q$, Lemma \ref{last1} shows that $Q$ is among the minimal prime ideals of $\ann_{R}R^{\alpha}/((P+aR)R^{\alpha})^{\star U}$.
			\item If $a\notin Q$, we shall apply Lemma \ref{last2} with the matrix $X$ with entries in $R_{a}$ such that the $\alpha$-th elementary vector $e_{\alpha}\in W_{1}=XW_{a}\cap R^{\alpha}$ and $U_{1}$ as in Lemma \ref{last2}. Then $R^{\alpha}/W_{1}\cong R^{\alpha-1}/\rho(W_{1})$, and so $Q$ is a minimal prime $\ann_{R} R^{\alpha-1}/\rho(W_{1})$. Let $U_{2}$ be the top left $(\alpha-1)\times(\alpha-1)$ submatrix of $U_{1}$. Then since $U_{2}\rho(W_{1})\subseteq \rho(U_{1}W_{1})\subseteq \rho(W_{1}^{[p]})=\rho(W_{1})^{[p]}$, $\ann_{R} R^{\alpha-1}/\rho(W_{1})$ is $U_{2}$-special, and so is $Q$. 
		\end{enumerate}
	\end{enumerate}  
	This shows that in any case $Q$ is an element of a finite set of prime ideals. Hence, there are only finitely many $U$-special prime ideals of $R$ which contain $P$ minimally.
\end{proof}

Next Theorem is a generalization of Theorem 5.5 in \cite{K3} to $R$, and it provides an effective algorithm for finding all $U$-special prime ideals $P$ of $R$ with the property that $\mathcal{K} \nsubseteq PR^{\alpha}$.

\begin{theorem}\label{last4}
	Let $P$ a $U$-special prime ideal of $R$ such that $\mathcal{K} \nsubseteq PR^{\alpha}$, and $Q$ be a $U$-special prime ideal minimally containing $P$. Let $M$ be a matrix whose columns generate $(PR^{\alpha})^{\star U}$.
	\begin{enumerate}
		\item If $PR^{\alpha} \subsetneq \im M$, then either
		\begin{enumerate}
			\item all entries of $M$ are in $Q$, and so there exist an element $a\in Q\setminus P$ and $Q$ is among the minimal prime ideals of $\ann_{R}R^{\alpha}/((P+aR)R^{\alpha})^{\star U}$, or
			\item there exists an entry of $M$ which is not in $Q$, and $Q$ is a special prime over an $(\alpha-1)\times(\alpha-1)$ matrix.
		\end{enumerate}
		\item If$PR^{\alpha}=\im M$, then there exist an element $a_{1}\in R\setminus P$, an element $g\in (P^{[p]}:P)$, and an $\alpha\times\alpha$ matrix $V$ such that for some $\mu\gg0$, we have $a_{1}^{\mu}U \equiv gV \text{ modulo } P^{[p]}$. If $d=\det V$, then either
		\begin{enumerate}
			\item $d\in P$, and $Q$ is a special prime ideal over an $(\alpha-1)\times(\alpha-1)$ matrix, or
			\item $d \in Q\setminus P$, and $Q$ is among the minimal prime ideals of $\ann_{R}R^{\alpha}/((P+dR)R^{\alpha})^{\star U}$, or
			\item $d \notin Q$, and $Q$ is a $g$-special ideal of $R$.
		\end{enumerate}
	\end{enumerate}
\end{theorem} 

\begin{proof} Let $W\subseteq R^{\alpha}$ be such that $UW\subseteq W^{[p]}$ and $Q=\ann_{R}R^{\alpha}/W$. When all entries of $M$ are in $P$, $\im M \subseteq PR^{\alpha}$, i.e., $\im M =(PR^{\alpha})^{\star U}=PR^{\alpha}$. Thus, if we are in case 1., we have at least one entry $a$ of $M$ which is not in $P$. If $a\in Q$, by Lemma \ref{last1}, $Q$ is among the minimal primes of $\ann_{R}R^{\alpha}/((P+aR)R^{\alpha})^{\star U}$. If $a\notin Q$, by Lemma \ref{last2}, $Q$ is a minimal prime of $\ann_{R}R^{\alpha}/W_{1}$ such that $U_{1}W_{1}\subseteq W_{1}^{[p]}$, where $U_{1}$ and $W_{1}$ as in Lemma \ref{last2}. On the other hand, since $a$ becomes a unit in $R_{a}$, we can choose the invertible matrix $X$ with entries in $R_{a}$ such that $W_{1}=XW_{a}\cap R^{\alpha}$ contains the $\alpha$-th elementary vector $e_{\alpha}$. Then we have $R^{\alpha}/W_{1}\cong R^{\alpha-1}/\rho(W_{1})$, where $\rho:R^{\alpha}\rightarrow R^{\alpha-1}$ is the projection onto first $\alpha-1$ coordinates. Let $U_{2}$ be the top left $(\alpha-1)\times(\alpha-1)$ submatrix of $U_{1}$. Then $\ann_{R}R^{\alpha}/W_{1}=\ann_{R}R^{\alpha-1}/\rho(W_{1})$ and $U_{2}\rho(W_{1})\subseteq \rho(U_{1}W_{1})\subseteq \rho(W_{1}^{[p]})=\rho(W_{1})^{[p]}$. Therefore, $\ann_{R}R^{\alpha}/W_{1}$ is $U_{2}$-special, and so is $Q$.
	
	Assume now that we are in case 2., by definition of $\star$-closure $UPR^{\alpha} \subseteq P^{[p]}R^{\alpha}$, i.e., the entries of $U$ are in $(P^{[p]}:P)$. On the other hand, by Lemma \ref{fedder}, if $A=R/P$, $F_{*}((P^{[p]}:P)/P^{[p]}) \cong\homm_{A}(F_{*}A,A)$ is rank one $F_{*}A$-module. This means that $(P^{[p]}:P)/P^{[p]}$ is rank one $A$-module, and so we can find an element $g \in (P^{[p]}:P)\setminus P^{[p]}$ such that $(P^{[p]}:P)/P^{[p]}$ is generated by $g+P^{[p]}$ as an $A$-module. Also we can find an element $a_{1} \in R\setminus P$ such that the localization of $(P^{[p]}:P)/P^{[p]}$ at $a_{1}$ is generated by $g/1+P_{a_{1}}^{[p]}$ as an $A_{a_{1}}$-module and hence as an $R_{a_{1}}$-module. If $a_{1}\in Q$, we can find $Q$ as in the case 1.(a), thus, we assume that $a_{1}\notin Q$. Then for any entry $u$ of $U$, working in the localization, we have an expression
	\[
	\frac{u}{1}+P_{a_{1}}^{[p]}=\frac{r}{a_{1}^{w_{1}}}\frac{g}{1}+P_{a_{1}}^{[p]}
	\]
	which implies that $\dfrac{u-rg}{a_{1}^{w_{1}}} \in P_{a_{1}}^{[p]}$, i.e., $\dfrac{u- rg}{a_{1}^{w_{1}}}=\dfrac{r'}{a_{1}^{w_{2}}}$, where $r \in R$, $r' \in P^{[p]}$ and $w_{1},w_{2} \in \N$. Thus,
	\[
	a_{1}^{w_{1}+w_{2}}u=a_{1}^{w_{2}}rg+a_{1}^{w_{1}}r'
	\]
	Therefore, we can write $a_{1}^{\mu}U=gV+V'$ for some $\mu\gg 0$ and $\alpha \times \alpha$ matrices $V$ and $V'$ with entries in $R$ and $P^{[p]}$, respectively. Then by Proposition \ref{rmklast}, we may replace $V'$ with the zero matrix, since $\Ie(V'R^{\alpha}) \subseteq PR^{\alpha}$. Let $d=\det V$. We now consider three cases:
	\begin{enumerate}
		\item If $d \in P$, then the determinant of $V$ in the fraction field $\mathbb{F}$ of $A$, say $\bar{d}$, will be zero. So we can find an invertible matrix $X$ with entries in $\mathbb{F}$ such that the last column of $VX^{-1}$ is zero, and so is $UX^{-1}$. Let $a_{2}$ is the product of all denominators of entries of $X$ and $X^{-1}$, i.e. the entries of $X$ and $X^{-1}$ are in $R_{a_{2}}$. If $a_{2}\in Q$, we can find $Q$ as in the case 1.(a) again, thus, we also assume that $a_{2}\notin Q$. Let $a=a_{1}a_{2}$. By Lemma \ref{last2}, $P$ and $Q$ are $U_{1}$-special prime ideals where $U_{1}=a^{\nu}X^{[p]}UX^{-1}$ whose last column is zero. Then since $PR^{\alpha}=(PR^{\alpha})^{\star U} \Leftrightarrow UPR^{\alpha}\subseteq P^{[p]}R^{\alpha}$, we also have
		\[
		U_{1}PR^{\alpha}=a^{\nu}X^{[p]}UX^{-1}PR^{\alpha}\subseteq a^{\nu}X^{[p]}UPR^{\alpha} \subseteq UPR^{\alpha} \subseteq P^{[p]}R^{\alpha}
		\]
		which implies $PR^{\alpha}=(PR^{\alpha})^{\star U_{1}}$. Hence, we can produce $Q$ as in Proposition \ref{last3}.
		\item If $d\in Q\setminus P$, then by Lemma \ref{last1}, $Q$ is among minimal prime ideals of $\ann_{R}R^{\alpha}/((P+dR)R^{\alpha})^{\star U}$.
		\item If $d\notin Q$, let $a=da_{1}$, $W=(QR^{\alpha})^{\star U}$ and $X=I_{\alpha}$ be the $\alpha\times\alpha$ identity matrix. Then by Lemma \ref{last2}, $Q$ is a minimal prime ideal of $\ann_{R}R^{\alpha}/W_{1}$ where $W_{1}=(QR^{\alpha})_{a}^{\star U}\cap R^{\alpha}$. By definition of $\star$-closure $(QR^{\alpha})_{a}^{\star U}=(QR_{a}^{\alpha})^{\star U}$ is the stable value of the sequence
		\begin{align*}
			L_{0}&=QR_{a}^{\alpha}\\
			L_{1}&=\Ie_{1}(UQR_{a}^{\alpha})+QR^{\alpha}=\Ie_{1}(gVQR_{a}^{\alpha})+QR_{a}^{\alpha}=\Ie_{1}(gQR^{\alpha})_{a}+QR_{a}^{\alpha}\\
			L_{2}&=\\
			\vdots&
		\end{align*}
		which also equals to $(QR^{\alpha})^{\star gI_{\alpha}}$. The third equality for $L_{1}$ is because of the fact that $\Ie_{e}(-)$-operation commutes with localization and $V$ is invertible. This implies that $\ann_{R}R^{\alpha}/W_{1}$ is $gI_{\alpha}$-special, and so is $Q$. Therefore, $Q$ is $g$-special and can be computed using the Katzman-Schwede algorithm, since $g \notin P^{[p]}$.
	\end{enumerate}
	This method also shows that for a given $U$-special ideal $P$, there are only finitely many $U$-special prime ideals minimally containing $P$.
\end{proof}

For the sake of integrity, we shall give the proof of Theorem \ref{1}. The main difference between our methods and the methods in \cite[Section 5]{K3} is that we do not use the aid of injective hulls of residue fields although our results are identical with the results in \cite[Section 5]{K3} over power series rings.

\begin{proof}[Proof of Theorem \ref{1}] The proof is by induction on $\alpha$. The case $\alpha=1$ is established in section \ref{section:Katzman-Schwede Algorithm}. Assume that $\alpha>0$ and the claim is true for $\alpha-1$. Since zero ideal is always a $U$-special prime ideal of $R$, we start with $0$ and use Theorem \ref{last4} to find $U$-special prime ideals minimally containing $0$. Continuing this process recursively gives us bigger $U$-special prime ideals at each steps. Therefore, since $R$ is of finite dimension, the number of steps in this process is bounded by the dimension of $R$. Hence, there are only finitely many $U$-special prime ideals with the desired property.
\end{proof}

Next we turn Theorem \ref{last4} into an algorithm which gives us a generalization of the Katzman-Zhang algorithm to $R$. Note also that over power series rings the following is identical with the Katzman-Zhang algorithm.

\subsection*{Intput:}
An $\alpha\times\alpha$ matrix $U$ with entries in $R$ such that $\mathcal{K}\neq 0$. 
\subsection*{Output:}
Set of all $U$-special prime ideals $P$ of $R$ with the property that $\mathcal{K} \nsubseteq PR^{\alpha}$.
\subsection*{Initialize:}
$\mathcal{A}_{R^{\alpha}}=\{0\}, \mathcal{B}=\emptyset$.
\subsection*{Execute the following:}
If $\alpha =1$, use the Katzman-Schwede Algorithm to find desired primes, put these in $\mathcal{A}_{R^{\alpha}}$, output $\mathcal{A}_{R^{\alpha}}$ and stop.\\
If $\alpha > 1$, then while $\mathcal{A}_{R^{\alpha}} \neq \mathcal{B}$, pick any $P \in \mathcal{A}_{R^{\alpha}}\setminus\mathcal{B}$. If $\mathcal{K}\subseteq PR^{\alpha}$, add $P$ to $\mathcal{B}$, if not, write $W=(PR^{\alpha})^{\star U}$ as the image of a matrix $M$ and do the following:
\begin{enumerate}
	\item If there is an entry $a$ of $M$ which is not in $P$, then;
	\begin{enumerate}
		\item Find the minimal primes of $\ann_{R}\dfrac{R^{\alpha}}{((P+aR)R^{\alpha})^{\star U}}$, and add them to $\mathcal{A}_{R^{\alpha}}$,
		\item Find an invertible $\alpha \times \alpha$ matrix $X$ with entries in $R_{a}$ such that the $\alpha$-th elementary vector $e_{\alpha}\in XW_{a}\cap R^{\alpha}$, and choose $\nu\gg 0$ such that $U_{1}=a^{\nu}X^{[p]}UX^{-1}$ has entries in $R$. Let $U_{0}$ be the top left $(\alpha-1)\times(\alpha-1)$ submatrix of $U_{1}$. Then apply the algorithm recursively to $U_{0}$ and add resulting primes to $\mathcal{A}_{R^{\alpha}}$.
	\end{enumerate}
	\item If $\im M=PR^{\alpha}$, then find elements $a_{1} \in R\setminus P$, $g \in (P^{[p]}:P)$, and an $\alpha \times \alpha$ matrix $V$, and $\mu \gg 0$ such that $a_{1}^{\mu}U \equiv gV$ modulo $P^{[p]}$. Compute $d= \det V$ and do the following:
	\begin{enumerate}
		\item If $d \in P$, find an element $a_{2} \in R \setminus P$ and an invertible matrix $X$ with entries in $R_{a_{2}}$ such that the last column of $UX^{-1}$ is zero. Find $\nu \gg 0$ such that the entries of $U_{1}=(a_{1} a_{2})^{\nu} X^{[p]}UX^{-1}$ are in $R$. Let $U_{0}$ be the top left $(\alpha-1)\times(\alpha-1)$ submatrix of $U_{1}$, and $\mathcal{K}_{0}$ be the stable value of $\{\Ie_{e}(\im U_{0}^{[p^{e}-1]}U_{0}^{[p^{e}-2]} \cdots U_{0})\}_{e>0}$ as in Theorem \ref{2}. Then;
		\begin{enumerate}
			\item If $\mathcal{K}_{0} \subseteq PR^{\alpha -1}$, write the last row of the matrix $U_{1}^{[p^{e-1}]}U_{1}^{[p^{e-2}]} \cdots U_{1}$ as $(g_{1}, \dots ,g_{\alpha -1}, 0)$ and apply the Katzman-Schwede Algorithm to the case $u=g_{i}$ for each $i$, and add resulting primes to $\mathcal{A}_{R^{\alpha}}$,
			\item If $\mathcal{K}_{0}\nsubseteq PR^{\alpha -1}$, find recursively all prime ideals for $U_{0}$ which contain $P$ minimally and denote their intersection with $\tau$. Compute $\Ie_{1}(U_{1}l(\tau \mathcal{K}_{0}))^{\star U_{1}}$, and write this as the image of a matrix $M'$. Find an entry $a'$ of $M'$ not in $P$. Now;
			\begin{enumerate}
				\item Add the minimal primes of $\ann_{R}\dfrac{R^{\alpha}}{((P+a'R)R^{\alpha})^{\star U_{1}}}$ to $\mathcal{A}_{R^{\alpha}}$,
				\item Find an invertible matrix $X$ with entries in $R_{a'}$ such that the $\alpha^{\text{th}}$ elementary vector $e_{\alpha}\in X(\im M')_{a'} \cap R^{\alpha}$. Find $\nu \gg 0$ such that $U_{2}=(a')^{v}X^{[p]}U_{1}X^{-1}$ has entries in $R$. Let $U_{3}$ be the top left $(\alpha-1)\times(\alpha-1)$ submatrix of $U_{2}$. Apply the algorithm recursively to $U_{3}$, and add resulting primes to $\mathcal{A}_{R^{\alpha}}$. 
			\end{enumerate}
		\end{enumerate}
		\item If $d \notin P$, then;
		\begin{enumerate}
			\item add the minimal primes of $\ann_{R}\dfrac{R^{\alpha}}{((P+dR)R^{\alpha})^{\star U}}$ to $\mathcal{A}_{R^{\alpha}}$,
			\item apply the Katzman-Schwede algorithm to the case $u=g$, and add resulting primes to $\mathcal{A}_{R^{\alpha}}$.
		\end{enumerate}
	\end{enumerate}
	\item Add $P$ to $\mathcal{B}$
\end{enumerate}
Output $\mathcal{A}_{R^{\alpha}}$ and stop.

Since all the operations used in the above algorithm are defined for localizations of $R$, we can apply our algorithm to any localization of $R$ at a prime ideal $\mathfrak{p}$. In the rest of this section, we investigate the relations between output sets of our algorithm applied to $R$ and $R_{\mathfrak{p}}$.

\begin{lemma} \label{lemmalast} Let $\mathcal{R}$ be $R$ or $R_{\mathfrak{p}}$ or $\widehat{R_{\mathfrak{p}}}$. $P$ is a $U$-special ideal of $R$ not contained in $\mathfrak{p}$ if and only if $P\mathcal{R}$ is a $U$-special ideal of $\mathcal{R}$.
\end{lemma}

\begin{proof} Let $P$ be a prime ideal of $R$. Then
	\begin{align*}
		P \text{ is } U\text{-special }&\Leftrightarrow P=\ann_{R}R^{\alpha}/(PR^{\alpha})^{\star U}\\
		& \Leftrightarrow P\mathcal{R}=\ann_{\mathcal{R}}\mathcal{R}^{\alpha}/(P\mathcal{R}^{\alpha})^{\star U} \Leftrightarrow P\mathcal{R} \text{ is } U\text{-special }
	\end{align*}
\end{proof}

Our next theorem gives the exact relation between the output sets $\mathcal{A}_{R^{\alpha}}$ and $\mathcal{A}_{R_{\mathfrak{p}}^{\alpha}}$ of our algorithm for $R$ and $R_{\mathfrak{p}}$, respectively.

\begin{theorem} \label{3} Let $U$ be an $\alpha \times \alpha$ matrix with entries in $R$. Our algorithm commutes with localization: if $\mathcal{A}_{R^{\alpha}}$ and $\mathcal{A}_{R_{\mathfrak{p}}^{\alpha}}$ are the output sets of our algorithm for $R$ and $R_{\mathfrak{p}}$, respectively, then
	\[
	\mathcal{A}_{R_{\mathfrak{p}}^{\alpha}}= \{ PR_{\mathfrak{p}} \mid P \in \mathcal{A}_{R^{\alpha}} \text{ and } P \subseteq \mathfrak{p} \}.
	\]
\end{theorem}

Before proving our claim we need a remark which we will use it in step 2. of the proof.

\begin{remark}
	Keeping the notations of above theorem, for any prime ideal $P$ of $R$, and any submodule $K$ of $R^{\alpha}$ we have the property that $K \subseteq PR^{\alpha} \Leftrightarrow K_{\mathfrak{p}} \subseteq PR_{\mathfrak{p}}^{\alpha}$. We already know that $K \subseteq PR^{\alpha}$ implies $K_{\mathfrak{p}} \subseteq PR_{\mathfrak{p}}^{\alpha}$. For the converse, suppose the contrary that there is an element $k=(k_{1}, \dots ,k_{\alpha})^{t} \in K\setminus PR^{\alpha}$ where $k_{i} \in R\setminus P$ for some $i$. Then there exists an element $s \in R\setminus \mathfrak{p}$ such that $sk \in PR^{\alpha}$, i.e. $sk_{i} \in P$. Since $P$ is prime, $k_{i} \in P$ or $s \in P$, which is impossible. Therefore, $K_{\mathfrak{p}} \subseteq PR_{\mathfrak{p}}^{\alpha}$ implies that $K \subseteq PR^{\alpha}$.
\end{remark}

\begin{proof} By Theorem \ref{al1}, the Katzman-Schwede Algorithm commutes with localization. Therefore, we can, and do, assume $\alpha >1$. Let $P$ be the prime ideal of $R$ in the initial step of our algorithm, and $R_{\mathfrak{p}}$ be a localization of $R$ at a prime ideal $\mathfrak{p}$ containing $P$. Since the stable value of $\{\Ie_{e}(U^{[p^{e-1}]}U^{[p^{e-2}]} \cdots UR_{\mathfrak{p}}^{\alpha})\}_{e\geq1}$ is equal to the stable value of $\mathcal{K}_{\mathfrak{p}}=\{\Ie_{e}(U^{[p^{e-1}]}U^{[p^{e-2}]} \cdots UR^{\alpha})R_{\mathfrak{p}}\}_{e\geq1}$. we have $\mathcal{K} \subseteq PR^{\alpha} \Leftrightarrow \mathcal{K}_{\mathfrak{p}}\subseteq PR_{\mathfrak{p}}^{\alpha}$. Since $\star$-closure commutes with localization, whenever we write $(PR^{\alpha})^{\star U}$ as the image of a matrix $M$ with entries in $R$, we can write $(PR_{\mathfrak{p}}^{\alpha})^{\star U}=(PR^{\alpha})^{\star U}R_{\mathfrak{p}}$ as the image of same matrix but working in $R_{\mathfrak{p}}$.
	\begin{enumerate}
		\item Since $a \notin P \Leftrightarrow a \notin PR_{\mathfrak{p}}$, $a$ is an entry of $M$ not in $(PR_{\mathfrak{p}}^{\alpha})^{\star U}$. Then, by Lemma \ref{sp2}, step 1.(a) commutes with localization. However, for step 1.(b), we can take the same matrix $X$ with entries in $R_{a}$ but working in $R_{\mathfrak{p}}$. Then while we do operations in $R_{\mathfrak{p}}$, we see that $e_{\alpha} \in X(\im M)_{a}\cap R^{\alpha}$ implies that $e_{\alpha} \in (X(\im M)_{a}\cap R^{\alpha})R_{\mathfrak{p}}\cong X(\im M)_{a}\cap R_{\mathfrak{p}}^{\alpha}$. Also $U_{1}=a^{\nu}X^{[p]}UX^{-1}$ has entries in $R$ (and in $R_{\mathfrak{p}}$) for the same $\nu \gg 0$. Therefore, we end up with the same matrix $U_{0}$.
		\item  We first note  that $(PR^{\alpha})^{\star U}=PR^{\alpha} \Leftrightarrow (PR_{\mathfrak{p}}^{\alpha})^{\star U}=PR_{\mathfrak{p}}^{\alpha}$. Therefore, if $(PR_{\mathfrak{p}}^{\alpha})^{\star U}=PR_{\mathfrak{p}}^{\alpha}$, we can have the same construction working in $R_{\mathfrak{p}}$, i.e., we can take $a_{1} \in R_{\mathfrak{p}}\setminus PR_{\mathfrak{p}}$, $g \in ((PR_{\mathfrak{p}})^{[p]}:PR_{\mathfrak{p}})$, $\alpha \times \alpha$ matrix $V$ for the same $\mu \gg 0$ such that $a_{1}^{\mu}U=gV$ modulo $(PR_{\mathfrak{p}})^{[p]}$ and compute $d=\det V$.
		\begin{enumerate}
			\item For any $r \in R$, we have the property that $r \in P \Leftrightarrow r \in PR_{\mathfrak{p}}$. Thus, if $d \in PR_{\mathfrak{p}}$, then we can have the same construction again, and so we can take $a_{2} \in R_{\mathfrak{p}}\setminus PR_{\mathfrak{p}}$ and the same invertible matrix $X$ with entries in $R_{a_{2}}$ (and in $(R_{\mathfrak{p}})_{a_{2}}\cong (R_{a_{2}})_{\mathfrak{p}}$) such that the last column of $UX^{-1}$ is zero, working in $R_{\mathfrak{p}}$. We also can take the same $\nu \gg 0$ such that the entries of $U_{1}=(a_{1}a_{2})^{\nu} X^{[p]}UX^{-1}$ are in $R$ (and in $R_{\mathfrak{p}}$), and $U_{0}$ to be the same matrix. In addition, since $\Ie_{e}(-)$ operation commutes with localization, if we do calculations in $R_{\mathfrak{p}}$, then the stable value of 
			\[
			\{\Ie_{e}(U_{0}^{[p^{e-1}]}U_{0}^{[p^{e-2}]} \cdots U_{0}R_{\mathfrak{p}})\}_{e>0}
			\]
			is going to equal to the stable value of  $$\{\Ie_{e}(U_{0}^{[p^{e-1}]}U_{0}^{[p^{e-2}]} \cdots U_{0}R)R_{\mathfrak{p}}\}_{e>0}$$ which is $\mathcal{K}_{0}R_{\mathfrak{p}}$. Now, since $\mathcal{K}_{0} \subseteq PR^{\alpha -1} \Leftrightarrow \mathcal{K}_{0}R_{\mathfrak{p}} \subseteq PR_{\mathfrak{p}}^{\alpha -1}$, we can do next:
			\begin{enumerate}
				\item Working in $R_{\mathfrak{p}}$, if $\mathcal{K}_{0}R_{\mathfrak{p}} \subseteq PR_{\mathfrak{p}}^{\alpha -1}$ we can write the last row of the matrix $U_{1}^{[p^{e-1}]}U_{1}^{[p^{e-2}]} \cdots U_{1}$ as $(g_{1}, \dots ,g_{\alpha -1}, 0)$.
				\item Working in $R_{\mathfrak{p}}$, if $\mathcal{K}_{0}R_{\mathfrak{p}} \nsubseteq PR_{\mathfrak{p}}^{\alpha -1}$, we can apply our algorithm recursively to $U_{0}$ and find all prime ideals which contain $PR_{\mathfrak{p}}$ minimally and denote their intersection with $\bar{\tau}$, which is $\tau R_{\mathfrak{p}}$, as we have showed all steps of algorithm commute with localization. Then we have
				\[
				\Ie_{1}(U_{1}\bar{l}(\bar{\tau} \mathcal{K}_{0}R_{\mathfrak{p}}))^{\star U_{1}}=(\Ie_{1}(U_{1}l(\tau K))^{\star U_{1}})R_{\mathfrak{p}},
				\] 
				where $\bar{l}:R_{\mathfrak{p}}^{\alpha -1} \rightarrow R_{\mathfrak{p}}^{\alpha -1}\oplus R_{\mathfrak{p}}$ is the extension map induced by $l$.
			\end{enumerate}   
		\end{enumerate}
	\end{enumerate}
	All other steps are similar to previous steps, and so all steps of our algorithm commute with localization.
	
	Since our algorithm commutes with localization, by Lemma \ref{lemmalast}, the output set $\mathcal{A}_{R_{\mathfrak{p}}^{\alpha}}$ is the set of all $U$-special prime ideals of $R_{\mathfrak{p}}$, and hence, $$\mathcal{A}_{R_{\mathfrak{p}}^{\alpha}}= \{ PR_{\mathfrak{p}} \mid P \in \mathcal{A}_{R^{\alpha}} \text{ and } P \subseteq \mathfrak{p} \}.$$
\end{proof}

Let $U$ be an $\alpha\times\alpha$ matrix with entries in $R$, and let $\mathcal{A}_{R^{\alpha}}$ and $\mathcal{A}_{S^{\alpha}}$ be the output sets of our algorithm for $R$ and $S$, respectively. Let $P$ be a $U$-special prime ideal of $R$, i.e. $P\in\mathcal{A}_{R^{\alpha}}$. Since $PS$ is not always a prime ideal of $S$, we do not have a relation between $\mathcal{A}_{R^{\alpha}}$ and $\mathcal{A}_{S^{\alpha}}$ like in Theorem \ref{3}. However, by Lemma \ref{lemmalast}, we can say that the minimal prime ideals of $PS$ are in $\mathcal{A}_{S^{\alpha}}$. Therefore, the set of minimal prime ideals of elements from $\{ PS \mid P\in\mathcal{A}_{R^{\alpha}} \}$ is contained in $\mathcal{A}_{S^{\alpha}}$.

\section{An Application to Lyubeznik's F-modules}
\label{section:An Application to Lyubeznik's F-modules}
In this section, we investigate the connections between special ideals and local cohomology modules using Lyubeznik's theory of $F$-finite $F$-modules.

By Example \ref{t3}, the $i$-th local cohomology module of $R$ with respect to an ideal $I$ is an $F$-finite $F$-module and there exist a finitely generated module $M$ with an injective map $\beta:M\rightarrow F_{R}(M)$ such that
\[
H_{I}^{i}(R)=\varinjlim(M\xrightarrow{\beta}F_{R}(M) \xrightarrow{F_{R}(\beta)} F_{R}^{2}(M) \xrightarrow{F_{R}^{2}(\beta)}\cdots)
\]
where $\beta : M \rightarrow F_{R}(M)$ is a root morphism. Since $M$ is finitely generated, we also have $M \cong \coker A=R^{\alpha}/\im A$ for some matrix $A$ with entries in $R$. Hence,
\[
H_{I}^{i}(R)\cong\varinjlim(\coker A \xrightarrow{U} \coker A^{[p]}\rightarrow \cdots)
\]
for some $\alpha \times \alpha$ matrix $U$ with entries in $R$ such that $U\im A \subseteq \im A^{[p]}$. Furthermore, $U$ defines an injective map on $\coker A$, since $\beta$ is a root morphism.

\begin{remark} \label{l2} Following \cite[Section 4]{L1}, if $(R,\mathfrak{m})$ is a local ring, for any $F$-finite $F$-module $\mathcal{M}$, there exists a smallest $F$-submodule $\mathcal{N}$ of $\mathcal{M}$ with the property that $\dim_{R}\supp\mathcal{M}/\mathcal{N}=0$. Hence, $\mathcal{M}/\mathcal{N}\cong E^{k}$ as $R$-modules for some $k\in\N$, where $E$ is the injective hull of the residue field of $R$.
\end{remark}

\begin{definition}\label{crk}
	If $R$ is local, we define the \index{corank}corank of an $F$-finite $F$-module $\mathcal{M}$ the number $k$ in Remark \ref{l2}, and denote it by $\crk\mathcal{M} =k$. 
\end{definition}

In Section 4 of \cite{L1}, Lyubeznik uses the theory of corank to shed more light on the notion of $F$-depth of a scheme in characteristic $p$, which is analogous to the notion of DeRham depth of a scheme in characteristic $0$. Following \cite[Section 4]{L1}, in equicharacteristic $0$ one can interpret the DeRham depth in terms of closed points only. Proposition 4.14 in \cite{L1} shows that in characteristic $p$ we can not interpret the $F$-depth of a scheme $Y$ in terms of closed points only. To show this Lyubeznik proves that there are only finitely many prime ideals $P$ of $A$ such that $\crk(H_{IA_{P}}^{i}(A_{P})) \neq 0$. Here  $Y=\spec B$, where $B$ is a finitely generated algebra over a regular local ring $S$,  $A=S[x_{1},\cdots,x_{n}]$ and $I$ is the kernel of the surjection $A\rightarrow B$. Our next theorem not only reproves this result but also gives us an effective way to compute desired prime ideals.

\begin{theorem}\label{m1}
	Let $I$ be an ideal of $R$ and $P \subset R$ a prime ideal. If $H_{IR_{P}}^{i}(R_{P})$ has non zero corank then $P$ is in the output of our algorithm introduced in section \ref{section:Katzman-Zhang Algorithm}, i.e.
	\[
	\crk(H_{IR_{P}}^{i}(R_{P})) \neq 0 \Rightarrow P \in \mathcal{A}_{R^{\alpha}}.
	\]
	for some $\alpha\times\alpha$ matrix $U$ with entries in $R$.	
\end{theorem} 

\begin{proof}
	Since $H_{IR_{P}}^{i}(R_{P}) \cong R_{p}\otimes_{R} H_{I}^{i}(R)$, we have
	\[
	H_{IR_{P}}^{i}(R_{P}) \cong \varinjlim(\coker A_{P} \xrightarrow{U_{P}} \coker A_{P}^{[p]}\rightarrow \cdots)
	\]
	where $A_{P}$ and $U_{P}$ are localizations of $A$ and $U$, respectively. We also have that $U_{P}$ defines an injective map on $\coker A_{P}$ since $U$ defines a root morphism for $H_{I}^{i}(R)$.
	
	$\crk(H_{IR_{P}}^{i}(R_{P})) \neq 0$ implies that there exists a proper $F_{R_{P}}$-submodule $\mathcal{N}$ of $H_{IR_{P}}^{i}(R_{P})$ such that $\dim_{R_{P}}\supp (H_{IR_{P}}^{i}(R_{P})/\mathcal{N})=0$. Since $H_{IR_{P}}^{i}(R_{P})$ is $F_{R_{P}}$-finite, we have
	\[
	\mathcal{N}= \varinjlim (N \rightarrow F_{R_{P}}(N) \rightarrow F_{R_{P}}^{2}(N) \rightarrow \cdots)
	\]
	where $N=\mathcal{N} \cap \coker A_{P}$ is an $R_{P}$-submodule of $\coker A_{P}$. Thus, $N \cong V/\im A_{P}$ for some submodule $V \subseteq R_{P}^{\alpha}$ such that $U_{P}V \subseteq V^{[p]}$. Then 
	\begin{align*} 
		H_{IR_{P}}^{i}(R_{P})/\mathcal{N} &\cong \varinjlim(\coker A_{P}/N \xrightarrow{U_{P}} F_{R_{P}}(\coker A_{P}/N) \rightarrow \cdots) \\
		&\cong \varinjlim(R_{P}^{\alpha}/V \xrightarrow{U_{P}} R_{P}^{\alpha}/V^{[p]} \rightarrow \cdots ). 
	\end{align*}
	Furthermore,
	\begin{align*}
		\dim_{R_{P}}\supp (H_{IR_{P}}^{i}(R_{P})/\mathcal{N})=0 &\Rightarrow \ass(H_{IR_{P}}^{i}(R_{P})/\mathcal{N})=\{PR_{P}\} \\ &\Rightarrow \ass (R_{P}^{\alpha}/V)=\{PR_{P}\} \\ &\Rightarrow \ann_{R_{P}}(R_{P}^{\alpha}/V) \text{ is $PR_{P}$-primary}
	\end{align*}
	Therefore, $\ann_{R_{P}}(R_{P}^{\alpha}/V)$ is $U_{P}$-special and so is $PR_{P}$ by Lemma \ref{sp1}, because it is the only minimal prime ideal of $\ann_{R_{P}}(R_{P}^{\alpha}/V)$, i.e. $PR_{P} \in \mathcal{A}_{R_{P}^{\alpha}}$. Then by Theorem \ref{3}, $P \in \mathcal{A}_{R^{\alpha}}$
\end{proof}

\begin{corollary} \label{m2} $\mathcal{C}_{R}:=\{P \in \mathcal{A}_{R^{\alpha}} \mid (\im A_{P}+PR_{P}^{\alpha})^{\star U_{P}}\neq R_{P}^{\alpha}\}$ is the set of all prime ideals of $R$ which satisfy $\crk(H_{IR_{P}}^{i}(R_{P}))\neq 0$
\end{corollary}

\begin{proof}
	By Theorem \ref{m1}, $\crk(H_{IR_{P}}^{i}(R_{P}))\neq 0$ implies that $PR_{P}$ is a $U_{P}$-special prime ideal of $R_{P}$ such that $PR_{P}=\ann_{R_{P}}(R_{p}^{\alpha}/W)$ for some proper submodule $W\subset R_{p}^{\alpha}$ where $\im A_{P} \subseteq W$ and $A_{P}$ as in Theorem \ref{m1}. Since $(\im A_{P}+PR_{P}^{\alpha})^{\star U_{P}}$ is the smallest submodule of $R_{P}^{\alpha}$ which satisfies $PR_{P}=\ann_{R_{P}}(R_{p}^{\alpha}/(\im A_{P}+PR_{P}^{\alpha})^{\star U_{P}})$, if $(\im A_{P}+PR_{P}^{\alpha})^{\star U_{P}}=R_{P}^{\alpha}$, then we have a contradiction with the existence of $W$. Hence, the set of primes ideals of $R$ which satisfy $\crk(H_{IR_{P}}^{i}(R_{P})) \neq 0$ is the set $\{P \in \mathcal{A}_{R^{\alpha}} \mid (\im A_{P}+PR_{P}^{\alpha})^{\star U_{P}}\neq R_{P}^{\alpha}\}$.
\end{proof}

Corollary \ref{m2} says that if we want to compute the prime ideals of $R$ which satisfy $\crk(H_{IR_{P}}^{i}(R_{P}))\neq 0$, we pick an element $P\in \mathcal{A}_{R^{\alpha}}$ and need to check whether $(\im A_{P}+PR_{P}^{\alpha})^{\star U_{P}}$ is equal to $R_{P}^{\alpha}$.

\subsection*{Acknowledgements}  I would like to thank my supervisor Mordechai Katzman for his support, guidance and patience throughout this project. Without his helpful
advice and insights this preprint would not be possible.

\end{document}